\documentclass[11pt, a4paper, final]{amsart} 
\usepackage[left = 2.3cm, right = 2.30cm, top = 2.5cm, bottom = 2.5cm]{geometry}

\usepackage[utf8]{inputenc}
\usepackage[english]{babel}
\usepackage[T1]{fontenc}
\usepackage{lmodern}

\usepackage{amsmath, amsfonts, amssymb,  mathabx, mathrsfs, mathtools, dsfont}
\usepackage{rotating}
\usepackage{subcaption}
\usepackage{graphicx}

\usepackage{hyperref}
\usepackage[capitalize,noabbrev]{cleveref}% add nameinlink to color also the name
\hypersetup{%
	pdftitle={A discrete data assimilation algorithm for the reconstruction of Gray--Scott dynamics},  									% subject of the document
	colorlinks=true, 								% false: boxed links; true: colored links
	linkcolor= red,									% color of internal links (change box color with linkbordercolor) 
	anchorcolor=black,
	citecolor= red,									% color of links to bibliography
	urlcolor = red									% color of external (url) links
	pdfauthor = {Tsiry Avisoa Randrianasolo},
	pdfcreator = {Tsiry Avisoa Randrianasolo},
	pdfkeywords = {Data assimilation, feedback control, reaction-diffusion equations, Gray--Scott model, pattern formation, multigrid methods, finite volume methods}
}
\newcommand{\la}{\langle}
\newcommand{\ra}{\rangle}
\newcommand{\dd}{\,\mathrm{d}}
\newcommand{\RR}{\mathbb{R}}
\newcommand{\mC}{\mathcal{C}}
\newcommand{\mT}{\mathcal{T}}
\newcommand{\e}{\mathrm{e}}
\newcommand{\NN}{\mathbb{N}}
\newcommand{\HH}{\mathbb{H}}
\newcommand{\VV}{\mathbb{V}}
\newcommand{\mI}{\mathcal{I}}

\newcommand{\tu}{\tilde{u}}
\newcommand{\tv}{\tilde{v}}

\newcommand{\ud}{\underline{d}}

\newcommand{\bmu}{\bar{\mu}}
\newcommand{\CNL}{C_{\mathrm{NL}}}
\newcommand{\I}{\mathtt{I}}
\newcommand{\II}{\mathtt{II}}
\newcommand{\III}{\mathtt{III}}
\DeclareMathOperator*{\esssup}{{esssup}}

\newtheorem{theorem}{Theorem}[section]
\newtheorem{lemma}[theorem]{Lemma}% 
\newtheorem{proposition}[theorem]{Proposition}% 
\newtheorem{definition}{Definition}[section]
\newtheorem{algo}{Algorithm}[section]
\numberwithin{equation}{section}
\Crefname{proposition}{Proposition}{Propositions}
\title[Reconstruction of Gray--Scott dynamics]{A discrete data assimilation algorithm for the reconstruction of Gray--Scott dynamics}
\author{Tsiry Avisoa Randrianasolo}
\address{Faculty of Mathematics, Bielefeld University, 33615 Bielefeld, Germany}
\email{trandria@math.uni-bielefeld.de}
\thanks{This work was funded by the Deutsche Forschungsgemeinschaft (DFG, German Research Foundation) -- Project-ID 317210226 - SFB 1283.}
\begin{document}
	%
	% ===========================================================================
	% ABSTRACT
	% ===========================================================================
	%
	\begin{abstract}
		The Gray--Scott model governs the interaction of two chemical species via a system of reaction-diffusion equations.
		Despite its simple form, it produces extremely rich patterns such as spots, stripes, waves, and labyrinths. That makes it ideal for studying emergent behavior, self-organization, and instability-driven pattern formation. It is also known for its sensitivity to poorly observed initial conditions. Using such initial conditions alone quickly leads simulations to deviate from the true dynamics. The present paper addresses this challenge with a nudging-based data assimilation algorithm: coarse, cell-averaged measurements are injected into the model through a feedback (nudging) term, implemented as a finite-volume interpolant.
		We prove two main results. (i) For the continuous problem, the nudged solution synchronizes with the true dynamics, and the $L^2$-error decays exponentially under conditions that tie  observation resolution, nudging gains, and diffusion. (ii) For the fully discrete semi-implicit finite-volume scheme, the same synchronization holds, up to a mild time-step restriction.
		Numerical tests on labyrinthine patterns support the theory. They show recovery of fine structure from sparse data and clarify how the observation resolution, the nudging gain, and the frequency of updates affect the decay rate.
	\end{abstract}
	
	
	
	\maketitle

	%
	% ===========================================================================
	% Introduction
	%============================================================================
	%
	\section{Introduction}\label{sec:intro}
	
	We consider the problem of reconstructing the dynamics of the \emph{Gray--Scott model} from sparse observations on a bounded domain $D \subset \RR^2$ with Lipschitz boundary. The model consists of a system of reaction-diffusion equations that characterizes the evolution of two interacting chemical species with concentrations $u \coloneqq u(t,x)$, and $v \coloneqq v(t,x)$ satisfying
	\begin{equation} \label{eq:gs}
		\left\{
		\begin{alignedat}{6}
			\partial_t u &= d_u \Delta u &&- uv^2 &&+ F(1 - u), 
			\\
			\partial_t v &= d_v \Delta v &&+ uv^2 &&- (F + k)v,
		\end{alignedat}
		\right.
	\end{equation}
	for $(t,x) \in (0,T)\times D$, supplemented with initial conditions $u(0,x) = u_0(x)$, and  $v(0,x) = v_0(x)$. The constants $d_u>0$ and $d_v>0$ are the diffusion coefficients for the chemical species, respectively; 
	$F>0$ represents the feed rate of the reactant $u$, and $k>0$ denotes the removal rate of the intermediate species $v$.
	In addition, we impose homogeneous Neumann boundary conditions on $\partial D$, modeling a closed or insulated system in which no material can enter or leave the domain. In biological terms, this models scenarios such as a Petri dish or a membrane-bounded cell where neither species can escape the domain.
	
		The Gray--Scott model, originally introduced by P.\ Gray and S.K.\ Scott in the 1980s, see e.g.\ \cite{gray1983autocatalytic}, is derived as a simplification of the Oregonator model of the Belousov--Zhabotinsky chemical oscillator. The Gray--Scott popularity as a model system grew significantly following the influential work of Pearson \cite{pearson1993complex}, who demonstrated through numerical simulations that even minimal reaction-diffusion systems like Gray--Scott can give rise to surprisingly complex and diverse patterns. As discussed in texts such as Epstein and Pojman's Introduction to Nonlinear Chemical Dynamics \cite{epstein1998introduction}, the model is also known to exhibit both deterministic chaos and spatial patterning.
	As such, beyond its chemical origins, the Gray--Scott system serves as a paradigmatic example in mathematical biology, physics, and materials science, where it illustrates emergent behavior, self-organization, and instability-driven pattern formation.

	Regarding the mathematical well-posedness. The Gray--Scott model belongs to a broader class of reaction-diffusion systems whose mathematical properties, such as existence, uniqueness, and regularity of solutions, have been extensively studied using both weak and strong solution frameworks.
	
	For initial data $u_0,v_0\in L^2(D)$, one can construct global-in-time weak solutions using classical techniques, including energy estimates, monotonicity methods, and compactness arguments (see, e.g., \cite{amann1989dynamic, pierre2010global}). These solutions are defined in the distributional sense and satisfy associated energy inequalities. When the nonlinearities satisfy structural conditions such as quasi-positivity, cooperativity, and subcritical growth, global existence is ensured under homogeneous Neumann or Dirichlet boundary conditions.
	
	If the initial data are further assumed to lie in $L^\infty(D)$, then the solutions exhibit improved regularity: they remain uniformly bounded in $L^\infty(D)$ for all times, and standard bootstrapping and maximum principle arguments show that they are also continuous in space and time. In this case, one obtains global strong solutions, which satisfy the system almost everywhere and possess additional spatial regularity.
	
	An intermediate case is when $u_0,v_0\in H^1(D)$. In this setting, the solution enjoys enhanced regularity, namely
	\[
	u, v \in \mC([0,T];L^2(D)) \cap L^2(0,T;H^1(D)).
	\]
	This follows from standard parabolic regularity theory, as shown, e.g., in Amann's analysis of quasilinear parabolic systems \cite[Theorem 2.1]{amann1989dynamic}, which demonstrates continuity in time and Sobolev regularity under minimal assumptions on the initial data. Although framed for general quasilinear systems, these results apply directly to the semilinear structure of the Gray--Scott model. A similar conclusion is shown in Pierre's survey \cite{pierre2010global}, where the regularizing effect of the diffusion operator is shown to yield such space-time regularity for initial data in $H^1(D)$.
	
	Even in the minimal case, where $u_0,v_0\in L^2(D)$, the smoothing property of the heat semigroup ensures that for any $t>0$, the solution becomes essentially bounded in space
	\[
	u, v \in L^\infty(D), \quad\mbox{ for all } t\geq t_0 >0,
	\]
	as established, for example, in \cite{souplet2006global}. Therefore, under general assumptions, bounded domain $D\subset\RR\times\RR$ with Lipschitz boundary, quasi-positive nonlinearities, and initial data in $L^\infty(D)$, one can show that
	\[
	u, v \in L^\infty(0,\infty;L^\infty(D)),
	\]
	with further regularity determined by the spatial dimension, the form of the nonlinearity, and the choice of boundary conditions. These regularity properties are not only central to the theoretical understanding of the system but also essential in the numerical modeling of the system.

	It is also known that the Gray--Scott system is very sensitive to initial conditions, which in experimental settings, are often impractical to measure. Observations are typically sparse, noisy, and confined to subregions of the domain, such as sensor locations or microscopic fields of view. As a result, we rarely have access to the full state of both species across the entire domain. If we use data obtained from partial observations to compute the solution to the system without correction, the simulation quickly deviates from the Truth, and we lose predictive power. That makes reliable forecasting challenging. This challenge has been recognized in control theory and numerical analysis, where frameworks for limited observation recovery have been developed \cite{jones1992determining, lions1971optimal}. Data assimilation addresses this by fusing such partial measurements with a predictive model to reconstruct the full system state. This capability is especially valuable in applications requiring real-time forecasting or control, such as guiding chemical pattern formation or predicting wave propagation in reaction-diffusion systems \cite{harlim2008filtering, temam1988infinite}.
	
	In recent decades, a variety of data assimilation algorithms have been developed to combine partial and noisy observations with PDE-based models, with the most widely used approaches including variational methods (3D-Var, 4D-Var) \cite{courtier1994a,dimet1986variational}, ensemble Kalman filters (EnKF and its variants) \cite{evensen1994sequential,evensen2009data}, particle filters \cite{doucet2001sequential,vanLeeuwen2009particle}, and hybrid ensemble-variational schemes \cite{bannister2017a,lorenc2003the}. Variational methods formulate assimilation as an optimization problem, often requiring adjoint PDE solvers (backward-in-time models used to compute gradients), while ensemble/particle-based approaches provide stochastic approximations of forecast uncertainty but depend on ensembles of model realizations (multiple PDE simulations to approximate error covariances) and, in the case of particle filters, repeated resampling (to avoid collapse of particle weights). Hybrid methods seek to balance flow-dependent ensemble statistics with the rigorous temporal structure of variational formulations and are now standard in operational weather forecasting. In contrast to these discrete and often statistically oriented frameworks, the continuous data assimilation   algorithm introduced by Azouani, Olson, and Titi \cite{azouani2014continuous} takes a control-theoretic perspective: observational information is incorporated directly into the PDE dynamics through a feedback, also called nudging term. This approach avoids the use of adjoints, ensembles, or resampling, while still providing rigorous convergence guarantees.% for a broad class of dissipative systems.

Despite the extensive study of the Gray--Scott model in pattern formation and nonlinear dynamics, its treatment within data assimilation remains largely unexplored. 
This work carries out data assimilation of the Gray--Scott model via the Azouani--Olson--Titi (AOT) nudging framework within a finite-volume discretization. AOT nudging has proved effective for other dissipative PDEs, including Navier--Stokes \cite{azouani2014continuous,farhat2019a,farhat2016abridged,olson2003determining}, Boussinesq and Rayleigh--B\'enard convection \cite{farhat2018assimilation,farhat2016data}, and magnetohydrodynamic/turbulence models \cite{foias2016a}, where it achieves synchronization under coarse observations. To the best of our knowledge, a rigorous treatment for Gray--Scott is missing, in part due to its specific nonlinear couplings.

Our objective is to formulate and analyze a reconstruction methodology for the Gray--Scott system based on a finite-volume interpolant. Within this setting we derive explicit conditions on observation resolution and nudging gains that guarantee synchronization, and we prove exponential decay of the $L^2$-error over time. Targeted numerical tests then corroborate these guarantees and reveal behavior not captured by our analysis.

	The paper is organized as follows. In \Cref{sec:prelim} we set notation, recall basic well-posedness for the Gray–Scott system, and state the weak formulations used throughout. \Cref{sec:discretization} builds the numerical setting and the data assimilation model: we describe the finite-volume grid, define the continuous nudging system, and present the discrete data assimilation algorithm used for forecasting and assimilation. \Cref{sec:main_results} contains the main theorems: well-posedness for the nudged problem, and synchronization at both the continuous and discrete levels, together with proofs. \Cref{sec:numerics} reports numerical experiments that validate the theory and explore how performance depends on observation resolution, nudging gains, and update frequency. Finally, \Cref{sec:conclusion} summarizes the findings and outlines directions for future work.

	%
	% ===========================================================================
	% Preliminaries
	%============================================================================
	%
	\section{Preliminaries}\label{sec:prelim}
	
	In this section, we introduce the mathematical framework and notations employed throughout the paper. We also recall key well-posedness results for the Gray--Scott system \eqref{eq:gs}.
	
	\subsection{Functional settings and notations}\label{subsec:settings}
	
	We denote by $L^p\coloneqq L^p(D)$, for $1 \leq p \leq \infty$, the usual Lebesgue spaces of measurable functions defined on $D$. The associated norm is written as
	\[
	\Vert u \Vert_{L^p}^p \coloneqq   \int_D \vert u(x)\vert^p \dd x  , \quad \text{for } 1 \leq p < \infty, \quad\Vert u \Vert_{L^\infty} \coloneqq \esssup_{x \in D} \vert u(x)\vert .
	\]
	
	For $s\in\NN$ and $1 \leq p \leq \infty$, the Sobolev space $W^{s, p}(D)$ consists of functions whose weak derivatives up to order $s$ belong to  $L^p(D)$. When $p = 2$, we write $H^s\coloneqq H^s(D)= W^{s,p}(D)$, which is a Hilbert space with inner product and norm
	\[
	((u,v))_\alpha\coloneqq \sum_{\vert\alpha\vert\leq s} (\partial^\alpha u,\partial^\alpha v),\quad\Vert u\Vert_{H^s}^2\coloneqq \sum_{\vert\alpha\vert\leq s} \Vert \partial^\alpha u\Vert ^2_{L^2} ,
	\]
	where $(\,\cdot\,,\,\cdot\,)$ denotes the standard inner product in $L^2$. 
	
	We consider solutions subject to homogeneous Neumann boundary conditions. Since such problems determine solutions only up to an additive constant, we work with the mean-zero subspaces
	$$
	\HH\coloneqq\bigg\{u\in L^2:\int_D u(x)\dd x  = 0\bigg\},\quad \VV\coloneqq\bigg\{u\in H^1:\int_D u(x)\dd x  = 0\bigg\},
	$$
	and let $\VV'\coloneqq H^{-1}$ denote  the dual space of $\VV$;   by $\la\,\cdot\,,\,\cdot\,\ra$, the duality pairing between $\VV$ and $\VV'$.
	
	For time-dependent functions, we write $L^p(0,T;X)$ to denote Bochner spaces of $X$-valued functions that are $L^p$-integrable in time, where $X$ is a Banach space.

	%
	% ===========================================================================
	% Weak Formulation and Analytical Background
	%============================================================================
	%
	\subsection{Weak Formulation}\label{subsec:weak_formulation}
	We now define the notion of weak solution for the Gray--Scott system \eqref{eq:gs} and state the assumptions under which existence and uniqueness are guaranteed.

	\begin{definition}[Weak solution]\label{def:weak_solution}
		Let $T>0$. A pair of functions $(u,v)$ such that
		\[
		u, v \in L^2(0,T; \VV) \cap L^\infty(0,T;\HH), \quad 
		\partial_t u, \partial_t v \in L^2(0,T; \VV')
		\]
		is a weak solution to the Gray--Scott system \eqref{eq:gs} if for all test functions $\phi, \psi \in \VV$ and almost every $t \in (0, T)$, the equations
		\begin{alignat}{5}
			\label{eq:weak_u}
			\la \partial_t u(t), \phi \ra + d_u (\nabla u(t), \nabla \phi ) 
			&=    (&-u(t)v^2(t) &+ F(1 - u(t)), \phi)& ,  
			\\
			\label{eq:weak_v}
			\la \partial_t v(t), \psi \ra + d_v  (\nabla v(t), \nabla \psi) 
			&= (&u(t)v^2(t) &- (F + k)v(t), \psi) &
		\end{alignat}
		hold with initial data satisfying
		\begin{equation}\label{eq:ic}
			u_0,v_0\in L^\infty,\quad\mbox{ with }\quad 0\leq u_0(x)\leq 1,\quad v_0(x)\geq 0\quad\mbox{a.e. in }D. 
		\end{equation}
	\end{definition}
	
	The condition \eqref{eq:ic} reflects the physical interpretation of $u$ and $v$ as concentrations, and ensures that the solution remains in a physically meaningful range.
	
	Under the above assumptions, global-in-time weak solutions exist and remain nonnegative and bounded. Specifically, one has
	\[
	0 \leq u(t,x) \leq 1 \quad \text{and} \quad v(t,x) \geq 0 \quad \text{for almost every } (t,x) \in [0,T] \times D,
	\]
	and moreover,
	\[
	u, v \in L^\infty(0,T;L^\infty).
	\]
	These results follow from the general theory of reaction-diffusion systems with quasi-positive nonlinearities and bounded initial data (see, e.g., \cite{amann1989dynamic,pierre2010global, souplet2006global}). In particular, the Gray--Scott system satisfies the structural assumptions required for the application of maximum principles, comparison techniques, and a priori bounds. The quasi-positivity of the reaction terms ensures that nonnegative initial data yields nonnegative solutions, while boundedness is propagated by the nonlinear structure and diffusion.
	
	Moreover, if the initial data $u_0, v_0 \in H^1$, then the solution enjoys enhanced regularity, notably
	\[
	u, v \in \mC([0,T];L^2) \cap L^2(0,T;H^1).
	\]
	This follows from classical parabolic regularity theory, relying on energy estimates and the smoothing effect of the diffusion operator.
	Uniqueness of weak solutions follows from standard monotonicity and Gronwall-type arguments (see, e.g., \cite{henry1981geometric, pierre2010global}).

	%
	%
	% ===========================================================================
	% Semi-implicit discretization by finite volume methods
	%============================================================================
	%
	%
	\section{Discrete data assimilation algorithm}\label{sec:discretization}
	
	In this section, we formulate the forecast and assimilation models for the Gray--Scott system~\eqref{eq:gs}. We first describe the finite-volume discretization of the Gray--Scott system and the associated semi-implicit finite-volume approximation. We then recall the nudging formulation at the continuous level and finally introduce the fully discrete data assimilation algorithm

	\subsection{Finite volume discretization of the spatial domain}\label{subsec:mesh}
	
	Let $\mT_h = \{K\}$ denote a finite collection of non-overlapping square control volumes of side length $h$, forming a uniform Cartesian mesh of $D$. Each cell $K \in \mathcal{T}_h$  is associated with a representative point $x_K \in K$, chosen as the geometric center of $K$, so that the cell average $u_K \approx u(x_K)$ for all $K \in \mathcal{T}_h$.
	
	To discretize the diffusion operator, we follow the finite volume framework outlined in \cite{eymard2000finite}. 
	Integrating the Laplacian over each control volume and applying the divergence theorem gives
	\[
	\int_K \Delta u \dd x = \int_{\partial K} \nabla u \cdot n_K \dd S \approx \sum_{\sigma \subset \partial K} \Phi_\sigma,
	\]
	where $n_K$ is the outward unit normal to $\partial K$, and $\Phi_\sigma$ is the numerical flux across the face $\sigma$. 
	
	We use the two-point flux approximation to define the flux as
	\[
	\Phi_\sigma = - \vert \sigma\vert \frac{u_L - u_K}{d_{KL}},
	\]
	where $u_K$ and $u_L$ are the values at the centers of adjacent cells $K$ and $L$, and $d_{KL} = \vert x_L - x_K\vert$. This leads to the discrete approximation
	\[
	\int_{\partial K} \nabla u \cdot n_K \dd S \approx - \sum_{\sigma \subset \partial K} \tau_\sigma (u_L - u_K),
	\]
	with transmissibility coefficient $\tau_\sigma =  {\vert\sigma\vert}/{d_{KL}}$. For a uniform mesh, $d_{KL} = h$.
	
	We define the discrete space $\HH_h \subset L^2(D)$ of piecewise constant functions,
	\[
	\HH_h \coloneqq \bigg\{ u_h \in\HH: u_h|_K = \bigg(\frac{1}{h^2}\int_K u(x)\dd x\bigg) \text{ for all } K \in \mathcal{T}_h\bigg \}.
	\]
	For $w_h, \phi_h \in \HH_h$, the discrete inner products are defined as 
	\[
	\big(w_h, \phi_h\big)_h \coloneqq h^2 \sum_{K \in \mathcal{T}_h} w_K \phi_K,\quad
	\big(\nabla w_h, \nabla \phi_h\big)_h \coloneqq \sum_{\substack{K, L \in \mathcal{T}_h \\ \sigma = K\mid L}} \tau_{\sigma} (w_K - w_L)(\phi_K - \phi_L),
	\]
	where $\sigma = K\mid L$ means that the face $\sigma$ shared by the two neighboring cells $K$ and $L$.

	Given this setup, we define the cellwise average  interpolant $\mI_h: \VV\rightarrow \HH_h$, such that
	\begin{equation*}
		\mI_h u\coloneqq\sum_{K\in\mT_h}\bigg(\frac{1}{h^2}\int_K u(x)\dd x\bigg)\chi_{K},\quad\forall u\in \VV,
	\end{equation*}
	where $\chi_{K}$ denotes the characteristic function of $K$.
	
	To estimate the error introduced by the interpolant $\mI_h$, we recall the following result from finite element theory, which also holds in the finite volume framework, see \cite[Theorem 4.4.4]{brenner2008mathematical}:
	\begin{proposition}[Bramble--Hilbert lemma]\label{prop:bramble_hilbert_lemma}
		Let $K\subset D$ be a bounded Lipschitz domain of diameter $\mathrm{diam}(K)$ and $m\geq 0$ be an integer. Let $\phi:H^1(K)\rightarrow L^2(K)$ be a bounded linear functional that vanishes on all polynomials in $\mathbb{P}_m$ (i.e., of degree $\leq m$). Then, for all $w\in H^s(K)$,
		\begin{equation*}
			\Vert \phi(w)\Vert_{L^2(K)}\leq C\mathrm{diam}(K)^{s-m}\Vert w\Vert_{H^s(K)},
		\end{equation*}
		with a constant $C$ that depends only on the shape of $K$ and the choice of the functional $\phi$.
	\end{proposition}
	
	Using this, we state the following interpolation estimate:
	\begin{lemma}\label{lem:approx_pty}
		For any $u\in H^1$, the following inequality holds
		\begin{equation*}
			\Vert u-\mI_h u\Vert_{L^2}^2\leq \gamma_0h^2\Vert\nabla u\Vert^2_{L^2},
		\end{equation*}
		where the constant $\gamma_0$ depends only on the reference cell geometry and the quadrature rule.
	\end{lemma}
	\begin{proof}
		Let $u\in \VV$. Because $\mI_Hu$ is a piecewise constant
		\begin{align*}
			\Vert u-\mI_h u\Vert_{L^2}^2&= \sum_{K\in\mT_h} \Big\Vert u - \frac{1}{h^2}\int_K u(y)\dd y \Big\Vert_{L^2(K)}^2.
		\end{align*}
		Since the interpolant integrates constants exactly, the associated error functional
		\begin{equation*}
			\phi(u)\coloneqq u - \frac{1}{h^2}\int_K u(x)\dd x
		\end{equation*}
		vanishes on the constants of each $K$. Thus, Proposition~\ref{prop:bramble_hilbert_lemma} applies directly, yielding
		\begin{equation*} 
			\Big\Vert u - \frac{1}{h^2}\int_K u(x)\dd x \Big\Vert_{L^2(K)}^2\leq Ch^2\Vert u\Vert_{H^1(K)}^2.
		\end{equation*}
		Summing over $K$,  we obtain the global estimate, as claimed.
	\end{proof}
	
	\subsection{Continuous data assimilation}
	Using AOT nudging, we insert coarse, cell-averaged data into the model through the finite-volume interpolant.
	
	We define the nudged state $\tu\coloneqq\tu(t,x)$ and $\tv\coloneqq\tv(t,x)$  that satisfy the nudged system
	\begin{equation} \label{eq:nudgedgs}
		\left\{
		\begin{aligned}
			\partial_t{\tu} &= d_u\Delta \tu  -\tu\tv^2 + F(1-\tu) + \mu_u\big( \mI_H u -\mI_H \tu\big),
			\\
			\partial_t{\tv} &= d_v\Delta \tv  + \tu\tv^2 - (F+k)\tv + \mu_v\big(\mI_H v - \mI_H \tv\big),
		\end{aligned}
		\right.
	\end{equation}
	with initial data $\tu(0) = \tu_0\ge 0$ and $\tv(0) = \tv_0\ge 0$. To make this precise, we now define a weak formulation of the nudged Gray--Scott system driven by the sparse observations $\mI_H u$ and $\mI_H v$:
	
	\begin{definition}\label{def:weak_nudged_solution}
		Let $T>0$. Let $(\mI_Hu, \mI_H v)$ be the sparse observations we have on the Gray--Scott system \eqref{eq:gs}. A pair of functions $(\tu,\tv)$ such that
		\[
		\tu, \tv \in L^2(0,T; \VV) \cap L^\infty(0,T;\HH), \quad 
		\partial_t \tu, \partial_t \tv \in L^2(0,T; \VV')
		\]
		is a weak solution to the nudged Gray--Scott system \eqref{eq:nudgedgs} if for all test functions $\phi, \psi \in \VV$ and almost every $t \in (0, T)$, the equations
		\begin{alignat}{4}
			\label{eq:weak_nudged_u}
				 \la \partial_t \tu(t) ,\, \phi\ra
				+   d_u (\nabla \tu(t),\, \nabla \phi) & =  \big( &- \tu(t) (\tv(t))^2 &+ F(1-\tu(t)) &&+ \mu_u  ( \mI_H u(t) - \mI_H\tu(t) ),\, \phi \big)  ,
			\\%[5pt]
			\label{eq:weak_nudged_v}
				 \la \partial_t \tv(t) ,\, \psi \ra
				+   d_v (\nabla \tv(t),\, \nabla \psi)  &=   \big(  &\tu(t) (\tv(t))^2 &- (F+k)\tv(t) &&+ \mu_v  ( \mI_Hv(t) - \mI_H\tv(t) ),\, \psi \big) 
		\end{alignat}
		hold with initial data satisfying
		\begin{equation}\label{eq:ic_nudged}
			\tu_0,\tv_0\in L^\infty,\quad\mbox{ with }\quad 0\leq \tu_0(x)\leq 1,\quad \tv_0(x)\geq 0\quad\mbox{a.e. in }D. 
		\end{equation}
	\end{definition}
		
	The initial conditions are assumed to be arbitrary and generally unknown, reflecting the practical setting where the initial state of the system is not available. 
	The assimilation mechanism is designed to recover the true solution asymptotically from the sparse observations without requiring knowledge of the initial configuration. 
	%Although the feedback acts locally, diffusion and nonlinear coupling propagate its effect throughout the entire domain. 
	Two parameters $\mu_u$ and $\mu_v$, known as nudging (or feedback) gains, are there to control the strength of feedback.

	Under appropriate conditions on the observation resolution $H$, and a sufficiently large choice of $(\mu_u,\mu_v)$, the assimilated solution $(\tu,\tv)$ is expected to synchronize exponentially with the true state $(u,v)$ as $t\to\infty$.

		%
	% ===========================================================================
	% Fully discrete scheme
	%============================================================================
	%
	\subsection{Numerical models}\label{subsec:fully_discrete_scheme}
	
	We now define a numerical model of the Gray--Scott system.

	The time interval $[0,T]$ is divided into $N$ uniform time steps of size $\Delta t > 0$, with $t^n = n\Delta t$ for $n = 0, \dots, N$, where $t^0 = 0$, and $t^N = T$. At each time step $t^n$, let $u_h^n, v_h^n \in \VV_h$ denote the approximations of $u(t^n, \cdot)$ and $v(t^n, \cdot)$.
	
	\begin{algo}\label{algo:scheme}
		Given $u_h^{0}$, $v_h^{0}$. 
		For $n = 0,\ldots,N-1$: 
		
		Find $u_h^{n+1}, v_h^{n+1} \in \VV_h$ such that
		\begin{alignat}{6}
			\label{eq:discrete_gs_u}
			\big(u_h^{n+1} - u_h^n,\, &\phi_h \big)_h + \Delta t\, d_u \big(\nabla u_h^{n+1},\, &\nabla\phi_h\big)_h
			&= \Delta t \,\big(&-u_h^n (v_h^n)^2 &+ F(1 - u_h^n),\, &\phi_h &\big)_h,  
			\\
			\label{eq:discrete_gs_v}
			\big(v_h^{n+1} - v_h^n,\, &\psi_h\big)_h + \Delta t\, d_v \big(\nabla v_h^{n+1},\, &\nabla\psi_h\big)_h
			&= \Delta t\, \big(&u_h^n (v_h^n)^2 &- (F+k)v_h^n,\, &\psi_h &\big)_h, 
		\end{alignat}
		for all   $\phi_h, \psi_h \in \VV_h$.
	\end{algo}
	% ========================================================================================= 
	%
	%	
	We do not have real observational data for analysis or numerical experiments. Instead, we use \cref{algo:scheme} to generate synthetic data. Specifically, at each time step the discrete state $(u_h^n,v_h^n)$ obtained from \eqref{eq:discrete_gs_u}-\eqref{eq:discrete_gs_v} is regarded as the ``truth'', the reference solution we will compare everything against. % The coarse or sparse observational data $\mI_H u$ and $\mI_H v$   are then produced by applying the observation operator to this reference solution.
	
	We emphasize that the well-posedness, convergence, and a priori error estimates of  \cref{algo:scheme} have been rigorously analyzed in~\cite{randrianasolo2025convergence}, where it is shown that the scheme converges strongly to a weak solution of the continuous Gray--Scott system, with an $\mathcal{O}(h+\Delta t)$ error bound in the $L^2$-norm under suitable regularity assumptions.	We recall the following convergence result from \cite{randrianasolo2025convergence}:
	
	\begin{proposition}\label{prop:GS_imex_conv}
		As $h,\Delta t\to 0$, the sequences $u_h^{\Delta t}, v_h^{\Delta t}$ defined above converge to the functions $u,v$, which satisfy the weak formulation of the Gray--Scott system as defined by Definition~\ref{def:weak_solution}. The convergence
		is understood in the following sense:
		\begin{alignat*}{6}
			&u_h^{\Delta t}\to{} u,\;\; v_h^{\Delta t}\to{} v \quad\mbox{strongly in } L^2(0,T;\HH),
			\\
			&\nabla u_h^{\Delta t}\rightharpoonup {} \nabla u,\;\; \nabla v_h^{\Delta t}\rightharpoonup{} \nabla v \quad\mbox{weakly in } L^2(0,T;\HH\times \HH).
		\end{alignat*}
	\end{proposition}
	Thus, \cref{algo:scheme} provides both a mathematically controlled approximation of the continuous Gray--Scott dynamics and a consistent generator of synthetic observational data for the numerical tests.

	\subsection{Discrete data assimilation schemes}\label{subsec:fully_discrete_CDA}
	
	We now extend the fully discrete \cref{algo:scheme} to incorporate coarse observational data.  
	At each time step $t^n$, let $\tu_h^n, \tv_h^n \in \VV_h$ denote the approximations of the nudged state $(\tu(t^n,\cdot), \tv(t^n,\cdot))$. 
	
	The reference solution $(u_h^n,v_h^n)$ appearing in the nudging terms is generated by \cref{algo:scheme} and serves as the source of synthetic observational data. The observation operator $\mI_H$ projects fine-grid states  $(u_h^n,v_h^n)$ onto the coarser observation mesh by cell averaging.
	
	\begin{algo}\label{algo:nudged_scheme}
		Given $\tu_h^{0}$, $\tv_h^{0}$.
		For $n = 0,\ldots,N-1$: 

		Find $\tu_h^{n+1}, \tv_h^{n+1} \in \VV_h$ such that
		\begin{align}
			\label{eq:discrete_da_u}
			\begin{split}
				\big( \tu_h^{n+1}-\tu_h^n ,\, \phi_h \big)_h
				&+ \Delta t\, d_u (\nabla \tu_h^{n+1},\, \nabla \phi_h)_h
				\\
				&= \Delta t \big( - \tu_h^n (\tv_h^n)^2 + F(1-\tu_h^n) + \mu_u \big( \mI_H u_h^n - \mI_H\tu_h^n\big),\, \phi_h \big)_h,
			\end{split}
			\\[5pt]
			\label{eq:discrete_da_v}
			\begin{split}
				\big( \tv_h^{n+1}-\tv_h^n ,\, \psi_h \big)_h
				&+ \Delta t\, d_v (\nabla \tv_h^{n+1},\, \nabla \psi_h)_h
				\\
				&= \Delta t \big(  \hphantom{\; -\; }\tu_h^n (\tv_h^n)^2 - (F+k)\tv_h^n+ \mu_v \big( \mI_Hv_h^n - \mI_H\tv_h^n\big),\, \psi_h \big)_h,
			\end{split}
		\end{align}
		for all   $\phi_h, \psi_h \in \VV_h$.
	\end{algo}
	
	As in the continuous formulation, assimilation acts only within the coarse grid, but diffusion and nonlinear coupling propagate the correction throughout the whole domain. Under appropriate conditions on the observation resolution $H$, and a sufficiently large nudging gain $(\mu_u,\mu_v)$, the discrete nudged state $(\tu_h^n,\tv_h^n)$ is expected to converge to the truth $(u_h^n,v_h^n)$ as $n\to\infty$.
	
	%
	% ===========================================================================
	% Main results
	%============================================================================
	%
	\section{Main results}\label{sec:main_results}
	We have two main results. First, we prove synchronization for the \emph{continuous} data assimilation: the nudged Gray--Scott system drives the assimilated solution toward the true solution, and the error decreases over time. Second, we prove synchronization for the \emph{discrete} data assimilation: the fully discrete finite-volume scheme enjoys the same behavior, so the numerical solution also converges to the truth. We state each result, and provide proofs.

	\subsection{Well-posedness of the continuous data assimilation}\label{subsec:CDA_wellposedness}
	Let $ (\tu_h^n, \tv_h^n )\in \VV_h\times \VV_h $ be the solution to the nudged Gray--Scott system~\eqref{eq:nudgedgs} as computed by \cref{algo:nudged_scheme}.
	For $t\in (t^n,t^{n+1}]$, we define the piecewise linear interpolants
	\begin{equation*}
		\tu_h^{\Delta t}(t) \coloneqq \tu_h^n + \frac{\tu_h^{n+1} - \tu_h^n}{\Delta t}(t-t_n),\quad  \tv_h^{\Delta t}(t) \coloneqq \tv_h^n + \frac{\tv_h^{n+1} - \tv_h^n}{\Delta t}(t-t_n);
	\end{equation*}
	and the piecewise-constant time derivatives
	\begin{equation*}
		\partial_t \tu_h^{\Delta t}(t) \coloneqq \frac{\tu_h^{n+1} - \tu_h^n}{\Delta t}, \quad \partial_t \tv_h^{\Delta t}(t) \coloneqq \frac{\tv_h^{n+1} - \tv_h^n}{\Delta t}.
	\end{equation*}
	
	% THEOREM 1 *******************************************************************
	\begin{theorem}\label{thm:convergence}
		As $h,\Delta t\to 0$, the sequences $\tu_h^{\Delta t}, \tv_h^{\Delta t}$ defined above converge to the functions $\tu,\tv$, which satisfy the weak formulation of the nudged Gray--Scott system as defined by Definition~\ref{def:weak_nudged_solution}. The convergence
		is understood in the following sense:
		\begin{alignat*}{6}
		&\tu_h^{\Delta t}\to{} \tu, \quad  \tv_h^{\Delta t}\to{}\tv \quad\mbox{strongly in } L^2(0,T;\HH),
		\\
		&\nabla\tu_h^{\Delta t}\rightharpoonup{} \nabla\tu, \quad \nabla \tv_h^{\Delta t}\rightharpoonup{}\nabla \tv \quad\mbox{weakly in } L^2(0,T;\HH\times \HH).
		\end{alignat*}
	\end{theorem}

	The proof begins with auxiliary lemmas:
	Lemma~\ref{lem:discrete_nudged_unique_solution} shows well-posedness of the nudged step,
	and Lemmas~\ref{lem:discrete_nudged_bounded_solution}-\ref{lem:discrete_nudged_apriori} provide uniform bounds.
	These bounds yield compactness for the time-interpolated sequences,
	allowing us to pass to the limit in the discrete weak formulation and obtain the stated convergences.
	
	\medskip
	
	Let  $a_h^u, a_h^v : \VV_h \times \VV_h \rightarrow \mathbb{R}$ be bilinear forms defined as
	\[
	a_h^u(u,\phi_h) \coloneqq \big(u, \,\phi_h\big)_h + \Delta t\, d_u \big( \nabla u,\, \nabla\phi_h\big)_h, \quad
	a_h^v(v,\psi_h) \coloneqq \big(v, \,\psi_h\big)_h + \Delta t\, d_v \big(\nabla v,\, \nabla\psi_h\big)_h,
	\]
	and the associated linear functionals:
	\begin{alignat*}{4}
		\tilde\ell_h^u(\phi_h) &\coloneqq \big(\tu_h^n,\, \phi_h \big)_h + \Delta t\, \big(&-\tu_h^n (\tv_h^n)^2 &+ F(1 - \tu_h^n) + \mu_u \big( \mI_H u_h^n - \mI_H\tu_h^n\big),\, &\phi_h \big)_h, 
		\\
		\tilde\ell_h^v(\psi_h) &\coloneqq \big(\tv_h^n,\, \psi_h \big)_h + \Delta t\, \big(& \tu_h^n (\tv_h^n)^2 &- (F + k) \tv_h^n + \mu_v \big( \mI_Hv_h^n - \mI_H\tv_h^n\big), \,&\psi_h \big)_h,
	\end{alignat*}
	where $(u_h^n, v_h^n)$ is obtained from \cref{algo:scheme}.
	
	Then the variational formulation of the fully discrete problem reads:  
	
	Find $\tu_h^{n+1}, \tv_h^{n+1} \in \VV_h$ such that
	\begin{align}
		\label{eq:discrete_gs_vfu}
		a_h^u(\tu_h^{n+1}, \phi_h) &= \tilde\ell_h^u(\phi_h), \quad \forall \phi_h \in \VV_h,  
		\\
		\label{eq:discrete_gs_vfv}
		a_h^v(\tv_h^{n+1}, \psi_h) &= \tilde\ell_h^v(\psi_h), \quad \forall \psi_h \in \VV_h.
	\end{align}
	
	\begin{lemma}\label{lem:discrete_nudged_unique_solution}
		For $n = 0,\ldots, N-1$, the variational problem corresponding to the nudged scheme 
		\eqref{eq:discrete_da_u} (resp.\ \eqref{eq:discrete_da_v})  admits a unique solution 
		$\tu_h^{n+1}\in \VV_h$ (resp.\ $\tv_h^{n+1}\in \VV_h$).
	\end{lemma}
	
	\begin{proof}
		We again apply the Lax--Milgram theorem.
		
		\textit{Coercivity.} The bilinear forms are unchanged by nudging:
		\[
		a_h^u(u,u) = \Vert u\Vert _{L^2}^2 + \Delta t\, d_u \Vert \nabla u\Vert _{L^2}^2 \ge \Vert u\Vert _{\VV}^2,
		\]
		so $a_h^u(u,u)\ge \alpha \Vert u\Vert _{\VV}^2$ with $\alpha=1$. The same holds for $a_h^v$.
		
		\medskip
		
		\textit{Continuity.} As before,
		\[
		|a_h^u(u,\phi)| \le C \Vert u\Vert _{\HH}\Vert \phi\Vert _{\HH},
		\]
		for some $C>0$, and similarly for $a_h^v$.
		
		\medskip
		
		\textit{Boundedness of the linear forms $\tilde\ell_h^u$ and $\tilde\ell_h^v$.}  
		For the nudged right-hand sides, we estimate
		\[
		|\tilde\ell_h^u(\phi_h)|
		\le \Big( \Vert \tu_h^n\Vert _{L^2} 
		+ \Delta t \Vert  \tu_h^n (\tv_h^n)^2 \Vert _{L^2} 
		+ \Delta t F(1+\Vert \tu_h^n\Vert _{L^2})
		+ \Delta t \mu_u(\Vert \mI_H u_h^n\Vert _{L^2} + \Vert \mI_H \tu_h^n\Vert _{L^2})\Big)\Vert \phi_h\Vert _{L^2}.
		\]
		A similar estimate holds for $\tilde\ell_h^v$.  
		Thus $\tilde\ell_h^u,\tilde\ell_h^v$ are bounded linear functionals on $\VV_h$.
		
		\medskip
		
		Therefore, all conditions of the Lax--Milgram theorem are satisfied. 
		Hence there exists a unique solution $\tu_h^{n+1}\in \VV_h$ (resp.\ $\tv_h^{n+1}\in \VV_h$) 
		to \eqref{eq:discrete_da_u} (resp.\ \eqref{eq:discrete_da_v}), as claimed by the lemma.
	\end{proof}

	\begin{lemma}\label{lem:discrete_nudged_bounded_solution}
		Let $n\in\{0,\ldots,N-1\}$. Assume that $0\leq \tu_h^n\leq 1$ and $0\leq \tv_h^n\leq v_{\mathrm{max}}$ a.e.\ in $D$. Then, it holds for a.e.\ in $D$ that
		\[
		0\leq \tu_h^{n+1}\leq 1,\quad 0\leq \tv_h^{n+1}\leq v_{\mathrm{max}}.
		\]
	\end{lemma}
	
	\begin{proof}
		Fix $n\in\{0,\ldots,N-1\}$ and assume that $0\leq \tu_h^n\leq 1$ and $0\leq \tv_h^n\leq v_{\mathrm{max}}$ a.e.\ in $D$.  
		We again use the Stampacchia truncation argument applied to both $\tu_h^{n+1}$ and $\tv_h^{n+1}$.
		
		\textit{Non-negativity.}  
		Let $\tu_h^{n+1,-}\coloneqq\min\{0,\tu_h^{n+1}\}\in \VV_h$ and $\tv_h^{n+1,-}\coloneqq\min\{0,\tv_h^{n+1}\}\in \VV_h$.  
		Take $\phi_h=\tu_h^{n+1,-}$ in the variational form of \eqref{eq:discrete_da_u}, and $\psi_h=\tv_h^{n+1,-}$ in \eqref{eq:discrete_da_v}.  
		By coercivity of $a_h^u,a_h^v$, we obtain
		\begin{alignat*}{3}
			0 &\le a_h^u(\tu_h^{n+1,-},\tu_h^{n+1,-}) &&= a_h^u(\tu_h^{n+1},\tu_h^{n+1,-}) &&= \tilde\ell_h^u(\tu_h^n,\tu_h^{n+1,-}),
			\\
			0 &\le a_h^v(\tv_h^{n+1,-},\tv_h^{n+1,-}) &&= a_h^v(\tv_h^{n+1},\tv_h^{n+1,-}) &&= \tilde\ell_h^v(\tv_h^n,\tv_h^{n+1,-}).
		\end{alignat*}

		Recall that $\mI_H u_h^n$ and $\mI_Hv_h^n$ are non-negative. On the right-hand sides, observe that when $\tu_h^n=0$
		\begin{equation*}
			\tilde\ell_h^u(\tu_h^n,\tu_h^{n+1,-}) = \Delta t\big(F +\mu_u\mI_H u_h^n, \tu_h^{n+1,-}\big)_h,
		\end{equation*}
		and when $\tv_h^n=0$
		\begin{equation*}
			\tilde\ell_h^v(\tv_h^n,\tv_h^{n+1,-}) = \Delta t\mu_v\, \big( \mI_Hv_h^n , \,\tv_h^{n+1,-} \big)_h.
		\end{equation*}
		Also, since $\tu_h^{n+1,-}\le0$ and $\tv_h^{n+1,-}\le0$, these terms are non-positive.  
		Hence both right-hand sides are $\le0$, while the left-hand sides are $\ge0$, which forces
		$\Vert \tu_h^{n+1,-}\Vert _{L^2}=\Vert \tv_h^{n+1,-}\Vert _{L^2}=0$.  
		Thus, $\tu_h^{n+1},\tv_h^{n+1}\ge0$ a.e.\ in $D$.
		
		\textit{Upper bounds.}  
		Let $y_h^{n+1}=(\tu_h^{n+1}-1)^+\in\VV_h$ and $z_h^{n+1}=(\tv_h^{n+1}-v_{\mathrm{max}})^+\in\VV_h$.  
		Take $\phi_h=y_h^{n+1}$ in \eqref{eq:discrete_da_u} and $\psi_h=z_h^{n+1}$ in \eqref{eq:discrete_da_v}.  
		On the right-hand sides, observe that when $\tu_h^n=1$ 
		\begin{equation*}
			\tilde\ell_h^u(\tu_h^n,y_h^{n+1}) = \big(1,\, y_h^{n+1} \big)_h + \big(- \Delta t (\tv_h^n)^2 + \mu_u ( \mI_H u_h^n - 1),\, y_h^{n+1} \big)_h
		\end{equation*}
		and when $\tv_h^n=v_{\max}$
		\begin{equation*}
			\tilde\ell_h^v(\tv_h^n,z_h^{n+1}) = \big(v_{\max},\, z_h^{n+1}\big)_h + \Delta t\, \big(\tu_h^n v_{\max}^2 - (F + k) v_{\max} + \mu_v ( \mI_Hv_h^n - v_{\max}), \,z_h^{n+1} \big)_h.
		\end{equation*}
		Recall that $\mI_H u_h^n\leq 1$ and $\mI_Hv_h^n\leq v_{\max}$. Observe that the terms in the right-hand side are non-positive except $\big(1,\, y_h^{n+1} \big)_h$ and $\big(v_{\max},\, z_h^{n+1}\big)_h$, which we will send to the left-hand-side. Note that $\big(\tu_h^{n+1}-1,\, y_h^{n+1} \big)_h = \Vert y_h^{n+1}\Vert_{L^2}^2$ and $\big(\tv_h^{n+1}-v_{\max},\, z_h^{n+1} \big)_h = \Vert z_h^{n+1}\Vert_{L^2}^2$, we have
		\begin{align*}
			\Vert y_h^{n+1}\Vert_{L^2}^2 + d_u\Vert\nabla y_h^{n+1}\Vert_{L^2}^2\leq 0,
			\quad\mbox{ and } \quad
			\Vert z_h^{n+1}\Vert_{L^2}^2 + d_v\Vert\nabla z_h^{n+1}\Vert_{L^2}^2\leq 0.
		\end{align*}
		That imply $\Vert y_h^{n+1}\Vert _{L^2}=\Vert z_h^{n+1}\Vert _{L^2}=0$.  
		Therefore, $\tu_h^{n+1}\le1$ and $\tv_h^{n+1}\le v_{\max}$ a.e.\ in $D$.
		
		This concludes the proof.
	\end{proof}

	\begin{lemma}\label{lem:discrete_nudged_apriori}
		For all $N\in\NN$, the nudged scheme satisfies
		\begin{align*}
			\max_{0\leq n\leq N}\Big(\Vert \tu_h^n\Vert _{L^2}^2 + \Vert \tv_h^n\Vert _{L^2}^2 \Big)
			+ \Delta t \sum_{n=1}^N\Big(d_u\Vert \tu_h^n\Vert _{H^1}^2 + d_v\Vert \tv_h^n\Vert _{H^1}^2\Big) &\leq C,
			\\
			\sum_{n=1}^N\Big(\Vert \tu_h^n - \tu_h^{n-1}\Vert _{L^2}^2
			+ \Vert \tv_h^n - \tv_h^{n-1}\Vert _{L^2}^2\Big) &\leq C,
		\end{align*}
		with $C=C(D,F,v_{\max},T,\mu_u,\mu_v)>0$ independent of $h,\Delta t$.
	\end{lemma}
	
	\begin{proof}
		Take $\phi_h = 2\tu_h^{n+1}$ in \eqref{eq:discrete_da_u} 
		and $\psi_h = 2\tv_h^{n+1}$ in \eqref{eq:discrete_da_v}.  
		Using the identity 
		$(a-b)2a = a^2-b^2+(a-b)^2$, we obtain
		\begin{align*}
			\Vert \tu_h^{n+1}\Vert _{L^2}^2 - \Vert \tu_h^n\Vert _{L^2}^2 
			&+ \Vert \tu_h^{n+1}-\tu_h^n\Vert _{L^2}^2 
			+ \Delta t\, d_u\Vert \nabla \tu_h^{n+1}\Vert _{L^2}^2
			\\
			&= \Delta t\big(-\tu_h^n(\tv_h^n)^2 + F(1-\tu_h^n),\,2\tu_h^{n+1}\big)_h  + 2\Delta t\, \mu_u(\mI_H u_h^n - \mI_H \tu_h^n,\, \tu_h^{n+1})_h,
			\\[5pt]
			\Vert \tv_h^{n+1}\Vert _{L^2}^2 - \Vert \tv_h^n\Vert _{L^2}^2 
			&+ \Vert \tv_h^{n+1}-\tv_h^n\Vert _{L^2}^2 
			+ \Delta t\, d_v\Vert \nabla \tv_h^{n+1}\Vert _{L^2}^2
			\\
			&= \Delta t\big(\tu_h^n(\tv_h^n)^2 - (F+k)\tv_h^n,\,2\tv_h^{n+1}\big)_h + 2\Delta t\, \mu_v(\mI_H v_h^n - \mI_H \tv_h^n,\, \tv_h^{n+1})_h.
		\end{align*}
		
		By boundedness of $\tu_h^n,\tv_h^n$, see Lemma~\ref{lem:discrete_nudged_bounded_solution}, and the Cauchy--Schwarz and Young inequalities, the reaction contributions can be estimated as
		\begin{alignat*}{3}
			&\big(-&\tu_h^n(\tv_h^n)^2+F(1-\tu_h^n),2\tu_h^{n+1}\big)_h
			&\le C + \Vert \tu_h^{n+1}\Vert _{L^2}^2,
			\\[5pt]
			&\big(&\tu_h^n(\tv_h^n)^2-(F+k)\tv_h^n,2\tv_h^{n+1}\big)_h
			&\le C + \Vert \tv_h^{n+1}\Vert _{L^2}^2,
		\end{alignat*}
		with $C=C(D,v_{\max})>0$.
		
		For the nudging terms, we use Cauchy--Schwarz and Young inequalities to get
		\[
		2\mu_u(\mI_H u_h^n - \mI_H \tu_h^n,\, \tu_h^{n+1})
		\le \mu_u\Vert \mI_H u_h^n\Vert _{L^2}^2
		+ \mu_u\Vert \mI_H \tu_h^n\Vert _{L^2}^2
		+ \Vert \tu_h^{n+1}\Vert _{L^2}^2,
		\]
		and similarly for $v$.  
		These contributions are therefore controlled by the existing $L^2$ terms.
		
		Collecting these estimates gives
		\begin{align*}
			\Vert \tu_h^{n+1}\Vert _{L^2}^2 - \Vert \tu_h^n\Vert _{L^2}^2 
			+ \Vert \tu_h^{n+1}-\tu_h^n\Vert _{L^2}^2 
			+ \Delta t\, d_u\Vert \tu_h^{n+1}\Vert _{H^1}^2 
			&\le \Delta t C + \Delta t \Vert \tu_h^{n+1}\Vert _{L^2}^2,
			\\[5pt]
			\Vert \tv_h^{n+1}\Vert _{L^2}^2 - \Vert \tv_h^n\Vert _{L^2}^2 
			+ \Vert \tv_h^{n+1}-\tv_h^n\Vert _{L^2}^2 
			+ \Delta t\, d_v\Vert \tv_h^{n+1}\Vert _{H^1}^2 
			&\le \Delta t C + \Delta t \Vert \tv_h^{n+1}\Vert _{L^2}^2.
		\end{align*}
		
		Summing for $n=0,\dots,N-1$ and applying the discrete Gronwall inequality yields
		\begin{align*}
			\Vert \tu_h^N\Vert _{L^2}^2
			+ \sum_{n=1}^N\Vert \tu_h^n-\tu_h^{n-1}\Vert _{L^2}^2
			+ \Delta t\, d_u\sum_{n=1}^N\Vert \tu_h^n\Vert _{H^1}^2&\le C,
			\\
			\Vert \tv_h^N\Vert _{L^2}^2
			+ \sum_{n=1}^N\Vert \tv_h^n-\tv_h^{n-1}\Vert _{L^2}^2
			+ \Delta t\, d_v\sum_{n=1}^N\Vert \tv_h^n\Vert _{H^1}^2&\le C,
		\end{align*}
		with $C=C(D,F,v_{\max},T,\mu_u,\mu_v)>0$ independent of $h,\Delta t$.
		
		This completes the proof.
	\end{proof}

	\begin{proof}[Proof of \cref{thm:convergence}]
		Let $\tu_h^{\Delta t}, \tv_h^{\Delta t}$ denote the piecewise linear time interpolants of 
		$\tu_h^n, \tv_h^n\in \VV_h$. By the discrete energy estimate in 
		Lemma~\ref{lem:discrete_nudged_apriori}, we have 
		$\partial_t \tu_h^{\Delta t}, \partial_t \tv_h^{\Delta t}\in L^2(0,T;\HH_h)$.

		Interpolating the discrete problems 
		\eqref{eq:discrete_da_u}--\eqref{eq:discrete_da_v} in time, the interpolants 
		$\tu_h^{\Delta t}$ and $\tv_h^{\Delta t}$ satisfy
			\begin{align*}
			\begin{split}
				\int_0^T &\big( \partial_t \tu_h^{\Delta t}(s) ,\, \phi_h \big)_h
				+   d_u (\nabla \tu_h^{\Delta t}(s),\, \nabla \phi_h)_h \dd s
				\\
				&= \int_0^T \big( - \tu_h^{\Delta t}(s) (\tv_h^{\Delta t}(s))^2 + F(1-\tu_h^{\Delta t}(s)) + \mu_u  ( \mI_H u_h^{\Delta t}(s) - \mI_H\tu_h^{\Delta t}(s) ),\, \phi_h \big)_h \dd s,
			\end{split}
			\\[5pt]
			\begin{split}
				\int_0^T&\big( \partial_t \tv_h^{\Delta t}(s) ,\, \psi_h \big)_h
				+   d_v (\nabla \tv_h^{\Delta t}(s),\, \nabla \psi_h)_h\dd s
				\\
				&= \int_0^T  \big(  \hphantom{\; -\; }\tu_h^{\Delta t}(s) (\tv_h^{\Delta t}(s))^2 - (F+k)\tv_h^{\Delta t}(s)+ \mu_v  ( \mI_Hv_h^{\Delta t}(s) - \mI_H\tv_h^{\Delta t}(s) ),\, \psi_h \big)_h \dd s.
			\end{split}
		\end{align*}
		
			By the Aubin--Lions lemma, there exist subsequences (not relabeled) and limits 
		$u,v\in L^2(0,T;\HH)$ such that
		\begin{alignat}{7}
			\label{eq:strong_L2}
			&\tu_h^{\Delta t}\to{} \tu, \quad \tv_h^{\Delta t}\to{} \tv \quad\mbox{strongly in } L^2(0,T;\HH),
			\\
			\label{eq:weak_1}
			&\tu_h^{\Delta t}\rightharpoonup{} \tu, \quad \tv_h^{\Delta t}\rightharpoonup{} \tv \quad\mbox{weakly in } L^2(0,T;\VV),
			\\
			\label{eq:weak_2}
			&\partial_t \tu_h^{\Delta t}\rightharpoonup{}\partial_t \tu, \quad \partial_t \tv_h^{\Delta t}\rightharpoonup{}\partial_t \tv \quad\mbox{weakly in } L^2(0,T;\HH).
		\end{alignat}
		
		The weak converge \eqref{eq:weak_1}-\eqref{eq:weak_2} gives
		\begin{align*}
			\big( \partial_t \tu_h^{\Delta t}(s) ,\, \phi_h \big)_h
			+   d_u (\nabla \tu_h^{\Delta t}(s),\, \nabla \phi_h)_h&\to \big( \partial_t \tu(s) ,\, \phi \big)
			+   d_u (\nabla \tu(s),\, \nabla \phi),
			\\
			\big( \partial_t \tv_h^{\Delta t}(s) ,\, \psi_h \big)_h
			+   d_v (\nabla \tv_h^{\Delta t}(s),\, \nabla \psi_h)_h&\to  \big( \partial_t \tv(s) ,\, \psi \big)
			+   d_v (\nabla \tv(s),\, \nabla \psi).
		\end{align*}
		
		To pass to the limit in the nonlinear terms, note the decomposition
		\[
		\tu_h^{\Delta t} (\tv_h^{\Delta t})^2 - \tu \tv^2 
		= (\tu_h^{\Delta t}-\tu)(\tv_h^{\Delta t})^2 + \tu\big((\tv_h^{\Delta t})^2 - \tv^2\big).
		\]
		The boundedness, Lemma~\ref{lem:discrete_nudged_bounded_solution}, and strong convergence \eqref{eq:strong_L2} imply $\Vert \tu_h^{\Delta t}(
		\tv_h^{\Delta t})^2 - \tu\tv^2\Vert_{L^1(0,T;L^1)}\to 0$.
		
		By Proposition~\ref{prop:GS_imex_conv}, $(u_h^{\Delta t},v_h^{\Delta t})$ converges to $(u,v)$. Since $\mI_H$ is linear and continuous on $L^2$, we have 
		\begin{alignat}{7}
			\label{eq:IH_conv}
			&\mI_H u_h^{\Delta t}\to{} \mI_H u, \quad && \mI_H v_h^{\Delta t}\to{}  \mI_Hv \quad\mbox{strongly in } L^2(0,T;\HH),
		\end{alignat}
		which, with the strong convergence \eqref{eq:strong_L2} yield
		\begin{align*}
			\big( \mI_H u_h^{\Delta t}(s) - \mI_H\tu_h^{\Delta t}(s) ,\, \phi_h \big)_h &\to \big( \mI_H u(s) - \mI_H\tu(s),\, \phi \big), 
			\\ 
			\big( \mI_H v_h^{\Delta t}(s) - \mI_H\tv_h^{\Delta t}(s),\, \psi_h \big)_h &\to \big( \mI_H v(s) - \mI_H\tv(s),\, \psi \big).
		\end{align*}
		
		Passing to the limit in the weak formulation, we obtain
		\begin{align*}
			\begin{split}
				\int_0^T &\big( \partial_t \tu(s) ,\, \phi\big)
				+   d_u (\nabla \tu(s),\, \nabla \phi) \dd s
				\\
				&= \int_0^T \big( - \tu(s) (\tv(s))^2 + F(1-\tu(s)) + \mu_u  ( \mI_H u(s) - \mI_H\tu(s) ),\, \phi \big) \dd s,
			\end{split}
			\\[5pt]
			\begin{split}
				\int_0^T&\big( \partial_t \tv(s) ,\, \psi \big)
				+   d_v (\nabla \tv(s),\, \nabla \psi)\dd s
				\\
				&= \int_0^T  \big(  \hphantom{\; -\; }\tu(s) (\tv(s))^2 - (F+k)\tv(s)+ \mu_v  ( \mI_Hv(s) - \mI_H\tv(s) ),\, \psi \big) \dd s.
			\end{split}
		\end{align*}
		for all test functions $\phi,\psi\in\VV$.
		
		Therefore, $(\tu,\tv)$ is a weak solution of the nudged Gray--Scott system. 
		This completes the proof of \cref{thm:convergence}.
	\end{proof}

	\subsection{Synchronization at the continuous level}
	We analyze the synchronization at the continuous level by introducing the pointwise error functions
	\[
	e(t,x) \coloneqq \tu(t,x) - u(t,x), \quad f(t,x)\coloneqq \tv(t,x) - v(t,x),
	\]
	which represent the deviation of the assimilated state $(\tu,\tv)$ from the true solution $(u,v)$ of the Gray--Scott system.
	
	The following theorem establishes an exponential decay of the $L^2$-error under suitable conditions on the nudging strength $\mu$ and the observation resolution $H$.

	% THEOREM 2 *******************************************************************
	\begin{theorem}[Synchronization at the continuous level]\label{thm:CDA_conv}
		Let $\ud\coloneqq \min\{ d_u,d_v\}$, $\bmu\coloneqq \max\{ \mu_u,\mu_v\}$. 
		Provided that
		\begin{equation*}
			\ud>  \tfrac14\bmu\gamma_0 H^2,\quad \mbox{ and }\quad \bmu >{(\CNL - F)};
		\end{equation*}
		we have, for all $t>0$, an exponential decay of the error, 
		\begin{align*}
			\Vert e(t)\Vert_{L^2}^2 + \Vert f(t)\Vert_{L^2}^2 + \big(\ud- \tfrac14\bmu\gamma_0 H^2\big)\int_0^t\Big(\Vert \nabla e(s)\Vert_{L^2}^2  +  \Vert\nabla f(s)\Vert_{L^2}^2\Big) \dd s\leq  \big(\Vert e(0)\Vert_{L^2}^2 +\Vert f(0)\Vert_{L^2}^2\big) \e^{-2\gamma t},
		\end{align*}
		with a decay rate $\gamma \coloneqq F+\bmu- \CNL> 0$.
	\end{theorem}

 	We start the proof with a lemma:
	
	\begin{lemma}\label{eq:nl_estimate}
		We assume that
		$
		\Vert\tv\Vert^2_{L^\infty}\leq v_{\max},  \Vert v\Vert^2_{L^\infty}\leq v_{\max},   \Vert\tu\Vert^2_{L^\infty}\leq 1,   \Vert u\Vert^2_{L^\infty}\leq 1.
		$
		Define the differences
		$e \coloneqq \tu - u$ and $f\coloneqq \tv - v$. Then, it holds that
		\begin{equation*}
			 \big(\tu\tv^2 - uv^2 ,f - e\big) \leq \CNL\big(\Vert e\Vert^2_{L^2} + \Vert f\Vert^2_{L^2}\big),
		\end{equation*}
		where $\CNL\coloneqq 2( v_{\max}^2+     2v_{\max})$.
	\end{lemma}
	\begin{proof}
		We begin by expanding the nonlinear difference
		\begin{align*}
			\tu\tv^2 - uv^2 &= (\tu - u)\tv^2 + u(\tv^2 - v^2)= (\tu - u)\tv^2 + u(\tv - v)(\tv + v) = e\tv^2 + uf(\tv + v),
		\end{align*}
		so that
		$$
		  \big(\tu\tv^2 - uv^2 ,f - e\big)  =  \big( e\tv^2 ,f - e\big)  +   \big(uf (\tv + v ) ,f - e\big) .
		$$ 

		Let $\varepsilon >0$. To the first term in the right hand side, we apply the H\"older and Young inequalities,
		\begin{align*}
			 \big(e\tv^2 ,f - e\big) 
			&=   \big( e\tv^2 ,e\big)  +  \big(e\tv^2 ,f\big) 
			\leq \Vert\tv\Vert^2_{L^\infty}\big(\Vert e\Vert^2_{L^2}+  \Vert e\Vert_{L^2}\Vert f\Vert_{L^2}\big)
			\leq \Vert\tv\Vert^2_{L^\infty}\big( (1+\varepsilon)\Vert e\Vert^2_{L^2} + \tfrac1{4\varepsilon}\Vert f\Vert^2_{L^2}\big).
		\end{align*}
		We do the same for the second term,
		\begin{align*}
			&   \big(uf\big(\tv + v\big) ,f - e\big) =   \big(uf (\tv + v ) ,e\big)   +   \big( uf (\tv + v ) ,f\big)  
			\\
			&\leq\Vert u\Vert_{L^\infty}\Vert \tv + v\Vert_{L^\infty}\big(\Vert e\Vert_{L^2}\Vert f\Vert_{L^2} + \Vert f\Vert^2_{L^2}\big)
			 \leq \Vert u\Vert_{L^\infty}\Vert \tv + v\Vert_{L^\infty}\big(\tfrac1{4\varepsilon}\Vert e\Vert^2_{L^2} + (1+\varepsilon)\Vert f\Vert^2_{L^2}\big).
		\end{align*}
		
		Adding both contributions yields the estimate
		\begin{align*}
			  \big(\tu\tv^2 - uv^2 ,f - e\big)  \leq{}&  \big((1+\varepsilon)\Vert\tv\Vert^2_{L^\infty}   + \tfrac1{4\varepsilon}\Vert u\Vert_{L^\infty}\Vert \tv + v\Vert_{L^\infty} \big)\Vert e\Vert^2_{L^2}
			  \\
			  & + \big( \tfrac1{4\varepsilon}\Vert\tv\Vert^2_{L^\infty}   + (1+\varepsilon)\Vert u\Vert_{L^\infty}\Vert \tv + v\Vert_{L^\infty} \big)\Vert f\Vert^2_{L^2}.
		\end{align*}
		Now we fix $\varepsilon = 1$ to arrive at the result claimed by the lemma.
	\end{proof}
	
	\begin{proof}[Proof of \cref{thm:CDA_conv}]
			Subtracting the original system~\eqref{eq:gs} from the nudged system~\eqref{eq:nudgedgs}, we obtain the following PDE system governing the evolution of the error variables, for all test functions $\phi, \psi \in \VV$ and almost every $t \in (0, T)$,
		\begin{align}  
			\label{eq:gs_error_u}
			\begin{split}
				\la \partial_t{e}(t),\phi\ra &+ d_u\big(\nabla e(t), \nabla\phi\big) 
				\\
				&=   \big( u(t)v^2(t)-\tu(t)\tv^2(t) - Fe(t) + \mu_u\big( \mI_H u(t) -\mI_H \tu(t)\big), \phi\big), 
			\end{split}
			\\[5pt]
			\label{eq:gs_error_v}
			\begin{split}
				\la \partial_t{f}(t), \psi\ra &+ d_v\big(\nabla f(t),\nabla \psi\big)  
				\\
				&=  \big(\tu(t)\tv^2(t) - u(t)v^2(t) - (F+k)f(t)+ \mu_v\big( \mI_H v(t) -\mI_H \tv(t)\big), \psi\big).
			\end{split}
		\end{align}
		
		We fix $t \in (0, T)$.  We take $\phi = e$ in \eqref{eq:gs_error_u} and $\psi = f$ in \eqref{eq:gs_error_v} to get
		\begin{align*}
			\la \partial_t e(t),e(t)\ra + d_u\Vert\nabla e(t)\Vert_{L^2}^2 &= \big( u(t)v^2(t)-\tu(t)\tv^2(t)  - F  e(t) -\mu_u  \mI_H e(t), e(t)\big),
			\\[5pt]
			\la \partial_t f(t),f(t)\ra  + d_v\Vert\nabla f(t)\Vert_{L^2 }^2 &= \big(\tu(t)\tv^2(t) - u(t)v^2(t)- (F+k)  f(t) - \mu_v \mI_H f(t), f(t)\big).
		\end{align*}
		Adding the two equations, we obtain the energy balance
		\begin{align*}
			\frac12&\frac{\dd}{\dd t}\Big(\Vert e(t)\Vert_{L^2}^2 + \Vert f(t)\Vert_{L^2}^2 \Big) + d_u\Vert\nabla e(t)\Vert_{L^2}^2+d_v\Vert\nabla f(t)\Vert_{L^2}^2
			= \I(t) +\II(t)  + \III(t)
		\end{align*}
		with
		\begin{align*}
		\I(t) \coloneqq{} &- F  \Vert e(t)\Vert_{L^2}^2- (F +k) \Vert f(t)\Vert_{L^2}^2,
		\\[5pt]
		\II(t) \coloneqq {}&  -\mu_u\big(  \mI_H e(t), e(t)\big)-\mu_v\big(  \mI_H f(t), f(t)\big),
		\\[5pt]
		\III(t) \coloneqq{}&\hphantom{+}\big(\tu(t)\tv^2(t) - u(t)v^2(t) ,f(t) - e(t)\big).
		\end{align*}
		We skip the term $\I(t)$ as it is already dissipative in $L^2$. In what follows we estimate the terms $\II(t)$ and $\III(t)$ separately and then combine them back into the energy balance.
		
		 \textbf{Estimate for the nudging term $\II(t)$.} Let $\varepsilon>0$. Using the Cauchy--Schwarz and Young inequalities, we have
		\begin{align*}
			-\mu _u  \big(\mI_H e(t), e(t)\big)  &= \mu_u   \big(e(t)-\mI_H e(t), e(t)\big)  - \mu_u\Vert e(t)\Vert_{L^2}^2
			\\
			&\leq {\mu_u\varepsilon}\Vert  e(t)-\mI_H e(t)\Vert_{L^2}^2 -\mu_u\big(1-\tfrac{1}{4\varepsilon}\big)\Vert e(t)\Vert_{L^2}^2,
		\end{align*}
		and by Lemma~\ref{lem:approx_pty},
		\begin{align*}
			-\mu_u  \big( \mI_H e(t), e(t) \big) \leq {\mu_u\varepsilon\gamma_0 H^2}\Vert   \nabla e(t) \Vert_{L^2}^2 -\mu_u\big(1-\tfrac{1}{4\varepsilon}\big)\Vert e(t)\Vert_{L^2}^2.
		\end{align*}
		Analogously, through similar calculations, we obtain
		\begin{align*}
			-\mu_v  \big( \mI_H f(t), f(t)\big) \leq {\mu_v\varepsilon\gamma_0 H^2}\Vert  \nabla f(t) \Vert_{L^2}^2 -\mu_v\big(1-\tfrac{1}{4\varepsilon}\big)\Vert f(t)\Vert_{L^2}^2,
		\end{align*}
		which yield
		\begin{equation*}
			\II \leq {\varepsilon\gamma_0 H^2}\big(\mu_u\Vert   \nabla e(t) \Vert_{L^2}^2 +  \mu_v\Vert  \nabla f(t) \Vert_{L^2}^2\big) -\big(1-\tfrac{1}{4\varepsilon}\big)\big(\mu_u\Vert e(t)\Vert_{L^2}^2 +\mu_v\Vert f(t)\Vert_{L^2}^2\big).
		\end{equation*}
		
		\textbf{Estimate for the nonlinear term $\III(t)$.} We apply Lemma~\ref{eq:nl_estimate}, to get
		\begin{equation*}
			\vert  \big(\tu(t)\tv^2(t) - u(t)v^2(t) ,f(t) - e(t)\big) \vert \leq \CNL\big(\Vert e(t)\Vert^2_{L^2 } + \Vert f(t)\Vert^2_{L^2 }\big).
		\end{equation*}
%		\medskip

		 Combining the estimates for $\II(t)$ and $\III(t)$ back into the energy balance, we obtain
		\begin{align*}
			\frac12&\frac{\dd}{\dd t}\Big(\Vert e(t)\Vert_{L^2}^2 + \Vert f(t)\Vert_{L^2}^2 \Big) + (d_u - \mu_u\varepsilon\gamma_0 H^2)\Vert\nabla e(t)\Vert_{L^2}^2+(d_v - \mu_v\varepsilon\gamma_0 H^2)\Vert\nabla f(t)\Vert_{L^2}^2
			\\
			&\leq - \big(F + \mu_u\big(1-\tfrac{1}{4\varepsilon}\big) - \CNL\big)  \Vert e(t)\Vert_{L^2}^2- \big(F +k + \mu_v\big(1-\tfrac{1}{4\varepsilon}\big) - \CNL\big) \Vert f(t)\Vert_{L^2}^2.
		\end{align*}
		
		The right-hand side is dissipative for all $\varepsilon > 1/4$, if
		\begin{equation*}
			\mu_u >  (\CNL - F), \quad \mu_v > (\CNL - F-k).
		\end{equation*}
		
		We set $\ud \coloneqq\min\{d_u,d_v\}$, $\bmu \coloneqq\max\{\mu_u,\mu_v\}$, and $\gamma \coloneqq F+\bmu- \CNL$. By integration over time $t> 0$, we get
		\begin{align*}
			&\Vert e(t)\Vert^2_{L^2} + \Vert f(t)\Vert^2_{L^2}   + \big(\ud- \bmu\varepsilon\gamma_0 H^2\big)\int_0^t\Big(\Vert \nabla e(s)\Vert^2_{L^2}  +  \Vert\nabla f(s)\Vert^2_{L^2}\Big) \dd s
			\\
			&\leq  2\Vert e(0)\Vert^2_{L^2} + 2\Vert f(0)\Vert^2_{L^2} - 2\gamma\int_0^t\Big(\Vert e(s)\Vert^2_{L^2} + \Vert f(s)\Vert^2_{L^2}\Big)\dd s.
		\end{align*}
		We preserve the coercivity by imposing for all $\varepsilon \geq 1/4$, $\ud- \bmu\varepsilon\gamma_0 H^2 >0$. 
		Using the Gronwall inequality, we obtain the exponential decay
		\begin{align*}
			&\Vert e(t)\Vert^2_{L^2} + \Vert f(t)\Vert^2_{L^2} 
			+ \big(\ud- \bmu\varepsilon\gamma_0 H^2\big)\int_0^t\Big(\Vert \nabla e(s)\Vert^2_{L^2}  +  \Vert\nabla f(s)\Vert^2_{L^2}\Big) \dd s
			\\
			&\leq  2\big(\Vert e(0)\Vert^2_{L^2} +\Vert f(0)\Vert^2_{L^2}\big) \e^{-2\gamma t},
		\end{align*}
		with decay rate $\gamma$. 
		
		We fix $\varepsilon  =1/4$. Now the proof of the theorem is complete.
	\end{proof}

	% ********************************************************************************
	% PROOF OF THEOREM 2
	% ********************************************************************************
	\subsection{Synchronization at the discrete level}\label{subsec:error_estimate}
	
	 We analyze the synchronization at the discrete level by introducing the sequence of errors 
	 \[
	 e_h^n \coloneqq \tu_h^n - u_h^n, \quad f_h^n\coloneqq \tv_h^n - v_h^n,
	 \]
	 which represent the deviation of the discrete data assimilation $(\tu_h^n,\tv_h^n)$ from the computed data $(u_h^n,v_h^n)$ of the Gray--Scott system.
	 
	 \begin{theorem}[Synchronization at the discrete level]\label{thm:DDA_conv}
	 	Let $\ud\coloneqq \min\{ d_u,d_v\}$ and $\bmu\coloneqq \max\{ \mu_u,\mu_v\}$.
	 	Provided that 
	 	\begin{equation*}
	 		\ud>  \tfrac14\bmu\gamma_0 H^2,\quad \Delta t\leq { (\tfrac12(F+k+\bmu)+\CNL)^{-1}}, \quad\mbox{and}\quad \bmu >{\CNL - F};
	 	\end{equation*}
	 	we have, for all $m\in\NN$, an exponential decay of the error,
	 	\begin{align*}
	 		&\Vert  e_h^{m}\Vert_{L^2}^2+\Vert  f_h^{m}\Vert_{L^2}^2 + \big(1 - \Delta t\,( \tfrac12(F+k+\bmu)+\CNL)\big)\sum_{n=1}^{m}\big(\Vert  e_h^{n}-e_h^{n-1}\Vert_{L^2}^2+\Vert  f_h^{n}-f_h^{n-1}\Vert_{L^2}^2\big)
	 		\\
	 		& + 2\Delta t\, (\ud- \tfrac14\bmu\gamma_0H^2) \sum_{n=1}^{m}\big(\Vert \nabla e_h^{n}\Vert_{L^2}^2
	 		+ \Vert \nabla f_h^{n}\Vert_{L^2}^2\big)
	 		\leq  (\Vert  e_h^{0}\Vert_{L^2}^2 + \Vert  f_h^{0}\Vert_{L^2}^2)\e^{
	 			- 2\delta\Delta t m },
	 	\end{align*}
	 	with a decay rate $\delta \coloneqq F+\bmu-\CNL> 0$.
	 \end{theorem}
	 
	 \begin{proof}
	 	We fix $n\in\{0,\ldots,N-1\}$. Subtracting the discrete system~\eqref{eq:discrete_gs_u}-\eqref{eq:discrete_gs_v} from the discrete nudged system~\eqref{eq:discrete_da_u}-\eqref{eq:discrete_da_v}, we obtain   for all $\phi_h, \psi_h \in \VV_h$
	 	\begin{align}
	 		\label{eq:sync_disc_u}
	 		\begin{split}
	 			\big( e_h^{n+1}- e_h^n ,\, \phi_h \big)_h
	 			&+ \Delta t\, d_u (\nabla e_h^{n+1},\, \nabla \phi_h)_h
	 			\\
	 			&= \Delta t \big( u_h^n (v_h^n)^2  - \tu_h^n (\tv_h^n)^2 - F e_h^n - \mu_u  \mI_H e_h^n,\, \phi_h \big)_h,
	 		\end{split}
	 		\\[5pt]
	 		\label{eq:sync_disc_v}
	 		\begin{split}
	 			\big( f_h^{n+1}- f_h^n ,\, \psi_h \big)_h
	 			&+ \Delta t\, d_v (\nabla f_h^{n+1},\, \nabla \psi_h)_h
	 			\\
	 			&= \Delta t \big(  \tu_h^n (\tv_h^n)^2- u_h^n (v_h^n)^2 - (F+k)f_h^n - \mu_v \mI_Hf_h^n,\, \psi_h \big)_h.
	 		\end{split}
	 	\end{align}
	 	We take $\phi_h = 2 e_h^{n+1}$ in \eqref{eq:sync_disc_u} 
	 	and $\psi_h = 2 f_h^{n+1}$ in \eqref{eq:sync_disc_v}.  
	 	Using the identity $(a-b)2a = a^2-b^2+(a-b)^2$, we obtain
	 	\begin{align*}
	 		&\Vert  e_h^{n+1}\Vert_{L^2}^2 -\Vert  e_h^{n}\Vert_{L^2}^2 + \Vert  e_h^{n+1}-e_h^{n}\Vert_{L^2}^2 + 2\Delta t\, d_u\Vert \nabla e_h^{n+1}\Vert_{L^2}^2
	 		\\[5pt]
	 		&+\Vert  f_h^{n+1}\Vert_{L^2}^2 - \Vert  f_h^{n}\Vert_{L^2}^2 + \Vert  f_h^{n+1}-f_h^{n}\Vert_{L^2}^2 + 2\Delta t\, d_v\Vert \nabla f_h^{n+1}\Vert_{L^2}^2
	 		={} \I_n +\II_n + \III_n,
	 	\end{align*}
	 	with
	 	\begin{alignat*}{2}
	 		\I_n&\coloneqq{} -&&2\Delta t F\big(e_h^n, e_h^{n+1}\big)_h-2\Delta t (F+k)\big(f_h^n, f_h^{n+1}\big)_h,
	 		\\[5pt]
	 		\II_n &\coloneqq{}-&&2\Delta t\mu_u\big(\mI_H e_h^n,e_h^{n+1}\big)_h-2\Delta t\mu_v\big(\mI_H f_h^n,f_h^{n+1}\big)_h,
	 		\\[5pt]
	 		\III_n &\coloneqq{}&&2\Delta t \big(  \tu_h^n (\tv_h^n)^2- u_h^n (v_h^n)^2,\, f_h^{n+1} - e_h^{n+1} \big)_h.
	 	\end{alignat*}
	 	In what follows we estimate the terms $\I_n$, $\II_n$, and $\III_n$ separately and then combine them back into the energy balance.
	 		 	
	 	\textbf{Estimate for the first term $\I_n$.} Let $\varepsilon>0$. We split at the level $n+1$,
	 	\begin{align*}
	 		-\big(e_h^n, e_h^{n+1}\big)_h = \big(e_h^{n+1} - e_h^n, e_h^{n+1}\big)_h - \Vert e_h^{n+1}\Vert_{L^2}^2,
	 	\end{align*}
	 	and for $\varepsilon>0$, using the Young inequality, we have 
	 	\begin{align*}
		 		-\big(e_h^n, e_h^{n+1}\big)_h \leq \varepsilon\Vert e_h^{n+1} - e_h^n\Vert_{L^2}^2 -(1-\tfrac{1}{4\varepsilon})\Vert e_h^{n+1}\Vert_{L^2}^2.
	 	\end{align*}
	 		Analogously, through similar calculations, we obtain
	 		\begin{align*}
	 			-\big(f_h^n, f_h^{n+1}\big)_h \leq \varepsilon\Vert f_h^{n+1} - f_h^n\Vert_{L^2}^2 -(1-\tfrac{1}{4\varepsilon})\Vert f_h^{n+1}\Vert_{L^2}^2,
	 		\end{align*}
	 		which yield
	 	\begin{equation*}
	 		\I_n\leq  2\Delta t\, \varepsilon\big(F\Vert e_h^{n+1} - e_h^n\Vert_{L^2}^2 + (F+k)\Vert f_h^{n+1} - f_h^n\Vert_{L^2}^2\big) - 2\Delta t (1-\tfrac{1}{4\varepsilon} ) \big(F\Vert e_h^{n+1}\Vert_{L^2}^2 + (F+k)\Vert f_h^{n+1}\Vert_{L^2}^2\big).
	 	\end{equation*}
	 	
	 	\textbf{Estimate for the nudging term $\II_n$.} We split also the term $\big(\mI_H e_h^n,e_h^{n+1}\big)_h$ at the level $n+1$,
	 	\begin{align*}
	 		-\big(\mI_H e_h^n,e_h^{n+1}\big)_h= \big(\mI_H (e_h^{n+1} - e_h^n),e_h^{n+1}\big)_h- \big(\mI_H e_h^{n+1},e_h^{n+1}\big)_h.
	 	\end{align*}
	 	Using Lemma~\ref{lem:approx_pty}, we have
	 	\begin{align*}
	 		&-\big(\mI_H e_h^{n+1},e_h^{n+1}\big)_h = \big(e_h^{n+1}-\mI_H e_h^{n+1},e_h^{n+1}\big)_h - \Vert e_h^{n+1}\Vert_{L^2}^2
	 		\\[5pt]
	 		&\leq  \varepsilon\Vert e_h^{n+1}-\mI_H e_h^{n+1}\Vert_{L^2}^2- (1-\tfrac1{4\varepsilon})\Vert e_h^{n+1}\Vert_{L^2}^2
	 		\leq \varepsilon {\gamma_0 H^2} \Vert \nabla e_h^{n+1}\Vert_{L^2}^2 - (1-\tfrac1{4\varepsilon})\Vert e_h^{n+1}\Vert_{L^2}^2,
	 	\end{align*}
	 	and by Young inequality
	 	\begin{align*}
	 		\big(\mI_H (e_h^{n+1} - e_h^n),e_h^{n+1}\big)_h &\leq \varepsilon\Vert \mI_H (e_h^{n+1} - e_h^n)\Vert_{L^2}^2 + \tfrac1{4\varepsilon}\Vert e_h^{n+1}\Vert_{L^2}^2
	 		\leq  {\varepsilon}\Vert   e_h^{n+1} - e_h^n\Vert_{L^2}^2 + \tfrac1{4\varepsilon}\Vert e_h^{n+1}\Vert_{L^2}^2.
	 	\end{align*}
	 	We collect and obtain
	 	\begin{align*}
	 		-\big(\mI_H e_h^n,e_h^{n+1}\big)_h\leq   \varepsilon {\gamma_0 H^2} \Vert \nabla e_h^{n+1}\Vert_{L^2}^2 +  {\varepsilon}\Vert   e_h^{n+1} - e_h^n\Vert_{L^2}^2 -  (1-\tfrac1{2\varepsilon})\Vert e_h^{n+1}\Vert_{L^2}^2.
	 	\end{align*}
	 	By similar calculations, we obtain also  
	 	\begin{align*}
	 		-\big(\mI_H f_h^n,f_h^{n+1}\big)_h\leq   \varepsilon{\gamma_0 H^2}\Vert \nabla f_h^{n+1}\Vert_{L^2}^2 +\varepsilon\Vert   f_h^{n+1} - f_h^n\Vert_{L^2}^2- (1-\tfrac1{2\varepsilon})\Vert f_h^{n+1}\Vert_{L^2}^2.
	 	\end{align*}
	 	Thus, collecting these estimates, we obtain
	 	\begin{align*}
	 		\II_n\leq {}& 2\Delta t\, \varepsilon{\gamma_0 H^2} \big(\mu_u\Vert \nabla e_h^{n+1}\Vert_{L^2}^2 + \mu_v\Vert \nabla f_h^{n+1}\Vert_{L^2}^2\big) 
	 		\\
	 		&+  2\Delta t\,  \varepsilon\big(\mu_u\Vert   e_h^{n+1} - e_h^n\Vert_{L^2}^2 +\mu_v \Vert   f_h^{n+1} - f_h^n\Vert_{L^2}^2\big) 
	 		- 2\Delta t \, (1-\tfrac1\varepsilon)\big(\mu_u\Vert e_h^{n+1}\Vert_{L^2}^2 + \mu_v\Vert v_h^{n+1}\Vert_{L^2}^2\big).
	 	\end{align*}
	 	
	 	\textbf{Estimate for the nonlinear term $\III_n$.} 
	 	Observe that we can write
	 	\begin{align*}
	 		\tu_h^n (\tv_h^n)^2- u_h^n (v_h^n)^2 &= e_h^n (\tv_h^n)^2+ f_h^n u_h^n (\tv_h^n + v_h^n)  
	 		\\
	 		&= e_h^{n+1} (\tv_h^n)^2+ f_h^{n+1} u_h^n (\tv_h^n + v_h^n) - (e_h^{n+1} -e_h^n) (\tv_h^n)^2-  (f_h^{n+1} - f_h^{n}) u_h^n (\tv_h^n + v_h^n).
	 	\end{align*}
	 	Noting this, we split the nonlinear term $\III_n$ into 4 other terms,
	 	\begin{equation*}
	 		\III_n = \III_n^1 + \III_n^2 + \III_n^3 + \III_n^4,
	 	\end{equation*}
	 	with 
	 	\begin{align*}
	 			\III_n^1&\coloneqq 	2\Delta t\,\big(  e_h^{n+1} (\tv_h^n)^2,\, f_h^{n+1} - e_h^{n+1} \big)_h ,\quad 
	 			\\
	 			\III_n^2&\coloneqq 	2\Delta t\,\big(  f_h^{n+1} u_h^n (\tv_h^n + v_h^n),\, f_h^{n+1} - e_h^{n+1} \big)_h,
	 			\\
	 			\III_n^3&\coloneqq 2\Delta t\, \big( (e_h^{n} -e_h^{n+1}) (\tv_h^n)^2,\, f_h^{n+1} - e_h^{n+1} \big)_h,\quad 
	 			\\
	 			\III_n^4&\coloneqq 2\Delta t\,	\big(  (f_h^{n} - f_h^{n+1}) u_h^n (\tv_h^n + v_h^n),\, f_h^{n+1} - e_h^{n+1} \big)_h.
	 	\end{align*}
	 	Let $\eta >0$. On $\III_n^1$, we apply the H\"older and Young inequalities,
	 	\begin{align*}
	 		\III_n^1& = 2\Delta t\,\big(  e_h^{n+1} (\tv_h^n)^2,\, f_h^{n+1} \big)_h - 2\Delta t\,\big(  e_h^{n+1} (\tv_h^n)^2,\,  e_h^{n+1} \big)_h
	 		\\
	 		& \leq 2\Delta t\, \Vert \tv_h^n\Vert_{L^\infty}^2\big(\Vert e_h^{n+1}\Vert_{L^2}\Vert f_h^{n+1}\Vert_{L^2} +  \Vert e_h^{n+1}\Vert_{L^2}^2\big)
	 		 \leq 2\Delta t\, v_{\max}^2\big(( \varepsilon + 1)\Vert e_h^{n+1}\Vert_{L^2}^2 + \tfrac1{4\eta} \Vert f_h^{n+1}\Vert_{L^2}^2\big).
	 	\end{align*}
	 	The same on $\III_n^3$,
	 		\begin{align*}
	 			\III_n^3 & = 2\Delta t\,\big( (e_h^{n} -e_h^{n+1}) (\tv_h^n)^2,\, f_h^{n+1}   \big)_h - 2\Delta t\,\big( (e_h^{n} -e_h^{n+1}) (\tv_h^n)^2,\,   e_h^{n+1} \big)_h
	 			\\
	 			&\leq 2\Delta t\, \Vert \tv_h^n\Vert_{L^\infty}^2 \big(\Vert e_h^{n} -e_h^{n+1}\Vert_{L^2} \Vert f_h^{n+1}  \Vert_{L^2} + \Vert e_h^{n} -e_h^{n+1}\Vert_{L^2}  \Vert  e_h^{n+1}\Vert_{L^2}\big)
	 			\\
	 			&\leq 2\Delta t\, v_{\max}^2\big(2\eta\Vert e_h^{n} -e_h^{n+1}\Vert_{L^2}^2 +  \tfrac1{4\eta}  \Vert f_h^{n+1}  \Vert_{L^2}^2 +  \tfrac1{4\eta}  \Vert  e_h^{n+1}\Vert_{L^2}^2\big).
	 		\end{align*}
	 		By commuting $e_h$ and $f_h$;  and replacing $(\tv_h^n)^2$ by $ u_h^n (\tv_h^n + v_h^n)$ in $\III_n^1$ and $\III_n^3$, we obtain
	 		\begin{align*}
	 			\III_n^2 &\leq 2\Delta t\, (v_{\max}+ v_{\max})\big((1+\varepsilon)\Vert f_h^{n+1}\Vert_{L^2}^2 +  \tfrac1{4\eta}\Vert e_h^{n+1} \Vert_{L^2}^2\big),
	 			\\
	 			\III_n^4 &\leq 2\Delta t\, (v_{\max}+ v_{\max})\big(2\eta\Vert f_h^{n} -f_h^{n+1}\Vert_{L^2}^2 +  \tfrac1{4\eta}  \Vert e_h^{n+1}  \Vert_{L^2}^2 +  \tfrac1{4\eta}  \Vert  f_h^{n+1}\Vert_{L^2}^2\big).
	 		\end{align*}
	 		Collecting the above estimates for $\III_n^i$, $i = 1,\ldots,4$, we obtain
	 		\begin{align*}
	 		\III_n
	 		\leq{}  & 4\Delta t\, \eta \big(v_{\max}^2\Vert e_h^{n} -e_h^{n+1}\Vert_{L^2}^2 + (v_{\max}+ v_{\max})\Vert f_h^{n} -f_h^{n+1}\Vert_{L^2}^2  \big)
	 		\\
	 		& + 2\Delta t\,\big( (  \tfrac1{4\eta} + 1+\eta)v_{\max}^2 + \tfrac1{2\eta}(v_{\max}+ v_{\max}) \big) \Vert e_h^{n+1}\Vert_{L^2}^2
	 		\\
	 		& + 2\Delta t\,\big( \tfrac1{2\eta}v_{\max}^2 + (\tfrac1{4\eta} + 1 +\eta)(v_{\max}+ v_{\max}) \big) \Vert f_h^{n+1}\Vert_{L^2}^2.
	 		\end{align*}
 			We fix $\eta = 1/2$. We recall that $\CNL\coloneqq 2(v_{\max}^2 + (v_{\max}+ v_{\max}))$. Thus, we have
 			\begin{align*}
 				\III_n
 				\leq{}  & \Delta t\, \CNL \big(\Vert e_h^{n} -e_h^{n+1}\Vert_{L^2}^2 +\Vert f_h^{n} -f_h^{n+1}\Vert_{L^2}^2  \big)
 				 + 2\Delta t\,\CNL \big(\Vert e_h^{n+1}\Vert_{L^2}^2 + \Vert f_h^{n+1}\Vert_{L^2}^2\big).
 			\end{align*}
 			
 			Combining the estimates for $\I_n$, $\II_n$, and $\III_n$ back into the energy balance, we obtain
 			\begin{align*}
 				&\Vert  e_h^{n+1}\Vert_{L^2}^2 -\Vert  e_h^{n}\Vert_{L^2}^2+\Vert  f_h^{n+1}\Vert_{L^2}^2 - \Vert  f_h^{n}\Vert_{L^2}^2
 				\\[5pt]
 				&\quad+ \big(1-\Delta t\,(2\varepsilon(F+k+\bmu)+\CNL)\big)\big(\Vert  e_h^{n+1}-e_h^{n}\Vert_{L^2}^2 + \Vert  f_h^{n+1}-f_h^{n}\Vert_{L^2}^2\big)
 				\\[5pt]
 				&\quad  + 2\Delta t\, (\ud-\varepsilon\bmu {\gamma_0 H^2})\big(\Vert \nabla f_h^{n+1}\Vert_{L^2}^2 + \Vert \nabla e_h^{n+1}\Vert_{L^2}^2\big)
 				\\[5pt]
 				\leq {}& - 2\Delta t \Big[ \big((1-\tfrac{1}{4\varepsilon} )(F+\mu_u)-\CNL\big)\Vert e_h^{n+1}\Vert_{L^2}^2 + \big((1-\tfrac{1}{4\varepsilon} )(F+k+\mu_v)-\CNL\big)\Vert f_h^{n+1}\Vert_{L^2}^2\big)\Big],
 			\end{align*}
 			where we recall that $\ud \coloneqq\min\{d_u,d_v\}$, $\bmu \coloneqq\max\{\mu_u,\mu_v\}$. 
 			
 			For the right-hand side to remain be dissipative for all $\varepsilon \ge 1/4$, it suffices to choose
 			\begin{equation*}
 				\mu_u >\CNL  -F,\quad \mu_v >\CNL  -(F+k).
 			\end{equation*}
% 			We fix $\varepsilon >1/4$. Thus, if
% 			\begin{equation*}
% 				\mu_u >\Big(\frac{4\varepsilon}{4\varepsilon - 1}\Big)\CNL  -F,\quad \mu_v >\Big(\frac{4\varepsilon}{4\varepsilon - 1}\Big)\CNL  -(F+k),
% 			\end{equation*}
% 			then the right hand side is dissipative.
 			
 			Let  $\delta \coloneqq  (F+\bmu)-\CNL $. We sum over $m= 0,\ldots,n-1$, 
 			\begin{align*}
 				&\Vert  e_h^{m}\Vert_{L^2}^2+\Vert  f_h^{m}\Vert_{L^2}^2 + \big(1 - \Delta t\,(2\varepsilon(F+k+\bmu)+\CNL)\big)\sum_{n=1}^{m}\big(\Vert  e_h^{n}-e_h^{n-1}\Vert_{L^2}^2+\Vert  f_h^{n}-f_h^{n-1}\Vert_{L^2}^2\big)
 				\\
 				& + 2\Delta t\, (\ud-\varepsilon\bmu\gamma_0H^2) \sum_{n=1}^{N}\big(\Vert \nabla e_h^{n}\Vert_{L^2}^2
 				 + \Vert \nabla f_h^{n}\Vert_{L^2}^2\big)
 				\leq  \Vert  e_h^{0}\Vert_{L^2}^2 + \Vert  f_h^{0}\Vert_{L^2}^2
 				- 2\Delta t\delta   \sum_{n=1}^{m}\big( \Vert e_h^{n}\Vert_{L^2}^2 + \Vert f_h^{n}\Vert_{L^2}^2\big),
 			\end{align*}
 			and using the discrete Gronwall inequality, we obtain 
 			\begin{align*}
 				&\Vert  e_h^{m}\Vert_{L^2}^2+\Vert  f_h^{m}\Vert_{L^2}^2 + \big(1 - \Delta t\,(2\varepsilon(F+k+\bmu)+\CNL)\big)\sum_{n=1}^{m}\big(\Vert  e_h^{n}-e_h^{n-1}\Vert_{L^2}^2+\Vert  f_h^{n}-f_h^{n-1}\Vert_{L^2}^2\big)
 				\\
 				& + 2\Delta t\, (\ud-\varepsilon\bmu\gamma_0H^2) \sum_{n=1}^{m}\big(\Vert \nabla e_h^{n}\Vert_{L^2}^2
 				+ \Vert \nabla f_h^{n}\Vert_{L^2}^2\big)
 				\leq  (\Vert  e_h^{0}\Vert_{L^2}^2 + \Vert  f_h^{0}\Vert_{L^2}^2)\e^{
 				- 2\delta\Delta t m }.
 			\end{align*}

 			To preserve the coercivity, we impose for all $\varepsilon\ge 1/4$,
 			\begin{equation*}
 				\Delta t\leq (2\varepsilon(F+k+\bmu)+\CNL)^{-1},\quad \ud \ge \varepsilon\bmu\gamma_0H^2.
 			\end{equation*}
 			We choose $\varepsilon  = 1/4$, which completes the proof of the exponential decay with rate $\delta$.
	 \end{proof}

	%
	% ===========================================================================
	% Numerical experiments
	%============================================================================
	%
	\section{Numerical experiments}\label{sec:numerics}

	We present numerical experiments that validate the fully discrete synchronization predicted by \cref{thm:DDA_conv}. The discrete data assimilation \cref{algo:nudged_scheme} is implemented with the finite-volume package FiPy~\cite{fipy}, and the resulting linear systems are solved by the Preconditioned Conjugate Gradient (PCG) solver from SciPy.
	
	Because observations are available on a coarse grid while the model is discretized on a finer mesh, the nudging term must consistently transfer information between grids of different resolutions. We therefore employ restriction and prolongation operators, following standard multigrid or multi-resolution methods~\cite{hackbusch1985multi,trottenberg2001multi}, implemented on top of FiPy. The restriction operator projects fine-grid states onto the coarse observation grid by quadrature, averaging fine-grid values at coarse-cell quadrature points (here, a 4-point rule at the corners), which corresponds to volume-weighted restriction. 
	The prolongation operator lifts coarse data to the fine grid via nearest-neighbor interpolation, the simplest-possible choice.
	
	\begin{figure}[t]
		\begin{subfigure}[c]{1\textwidth}
			\begin{minipage}[t]{0.03\textwidth}
				\rotatebox{90}{ \small{Truth}}
			\end{minipage}
			\centering
			\includegraphics[scale = 0.37]{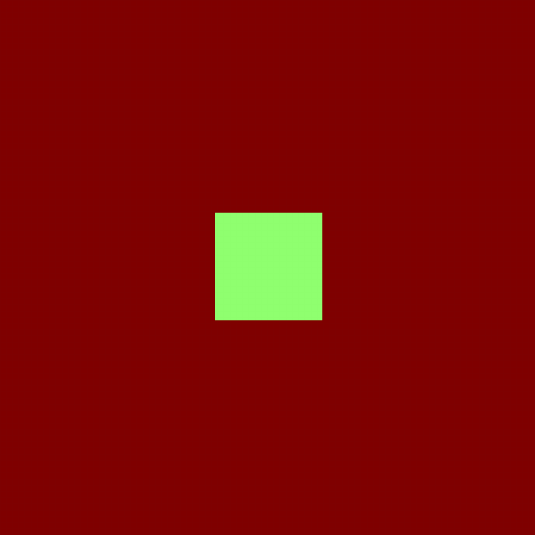}
			\includegraphics[scale = 0.37]{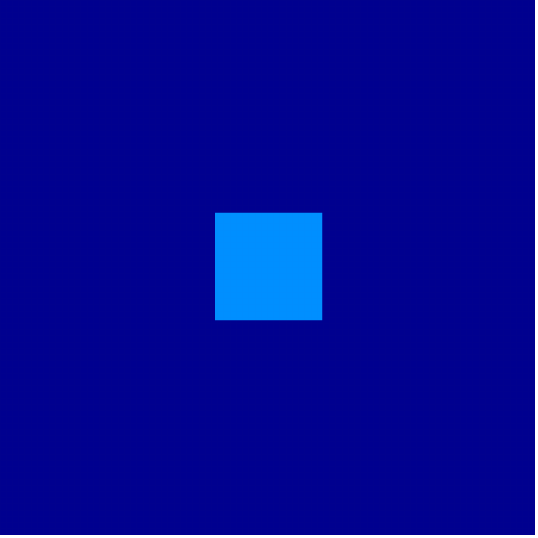}
		\end{subfigure}
		\\[5pt]
		\begin{subfigure}[c]{1\textwidth}
			\centering
			\begin{minipage}[t]{0.03\textwidth}
				\rotatebox{90}{ \small {Nudged}}
			\end{minipage}
			\includegraphics[scale = 0.37]{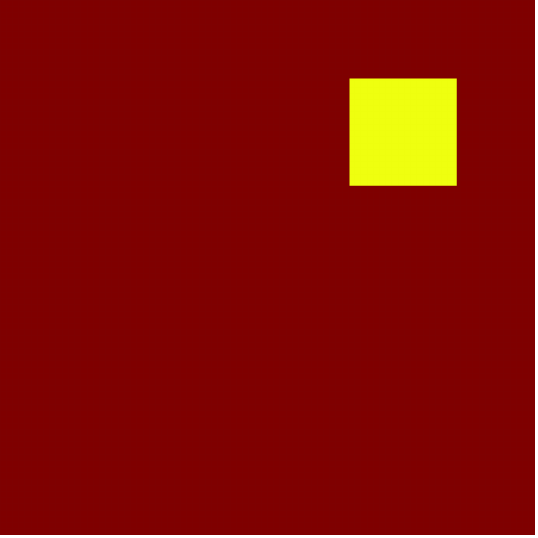}
			\includegraphics[scale = 0.37]{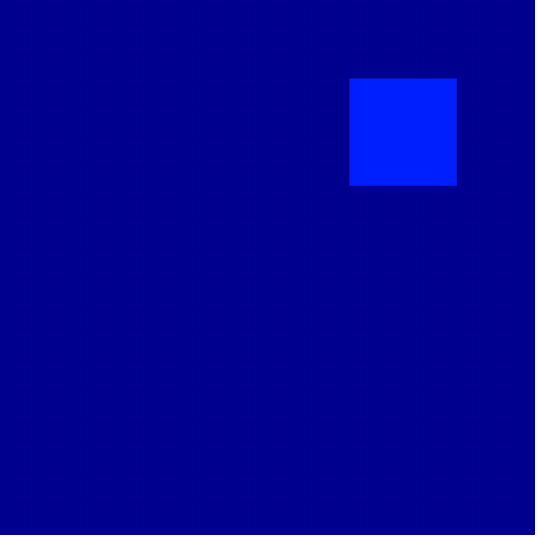}
		\end{subfigure}
		\hfill	
		\\[5pt]
		\begin{subfigure}[r]{1\textwidth}
			\centering
			\hphantom{\begin{minipage}[t]{0.019\textwidth}
					\rotatebox{90}{}
				\end{minipage}
			}
			\includegraphics[scale = 0.75]{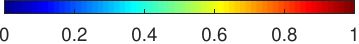}
		\end{subfigure}
		\caption{Initial conditions $u_0,\tu_0$ (left), and $v_0,\tv_0$ (right).}
		\label{fig:IC}
	\end{figure}

	We consider the Gray--Scott system on the unit square $D=[0,1]\times[0,1]$ with homogeneous Neumann boundary conditions for both species. The diffusion coefficients are fixed to $d_u=1.6\times10^{-5}$ and $d_v=d_u/2$.

	At $t=0$, see \cref{fig:IC}, the \emph{Truth} is initialized at the homogeneous steady state $u\equiv1$, $v\equiv0$, except on a central square of side $0.2$,
	\[
	Q_{\mathrm{c}}=[0.37,0.60]\times[0.37,0.60]\subset D,
	\]
	where $(u,v)=(0.50,0.25)$ to trigger pattern formation. 
	To model initial-state uncertainty, the \emph{nudged} system uses a misplaced seed of the same size but off center,
	\[
	Q_{\mathrm{off}}=[0.15,0.35]\times[0.60,0.80],
	\]
	with $(\tilde u,\tilde v)=(0.60,0.15)$ on $Q_{\mathrm{off}}$ and $(\tilde u,\tilde v)=(1,0)$ elsewhere. 
	This deliberate spatial and amplitude mismatch prevents trivial synchronization via coincident seeds.

	 Prior experiments indicate that observing only the intermediate species $v$ suffices to recover the full dynamics of the Gray--Scott system, whereas observing only the reactant $u$ does not yield synchronization. Accordingly, throughout numerical tests we set the feedback on the reactant $u$ to zero, $\mu_u = 0$.

	As a baseline, we consider the Gray--Scott system with feed rate	$F =	0.037$ and kill rate $k=0.060$. The time step is $\Delta t = 0.5$ unless varied. As time evolves, the reference ``Truth'' develops a \emph{labyrinthine} pattern, which appears with high resolution (240x240) in \cref{fig:labyrinthine} (first and second rows). \cref{fig:labyrinthine} (third row) shows low-resolution (24x24) observations on a coarser grid, representing the sparsity of available data that we have  on the intermediate species $v$. These observations are used in the discrete data assimilation \cref{algo:nudged_scheme} to reconstruct the full dynamics on a finer grid.
	
	%
	%
	% =======================================================
	\begin{figure}[t]
		\begin{subfigure}[t]{1\textwidth}
			
			\centering
			\makebox[.1\textwidth][c]{}
			\hspace{0pt}
			\makebox[.1\textwidth][c]{\small  $100$ tu}
			\hspace{50pt}
			\makebox[.1\textwidth][c]{\small  $1000$ tu}
			\hspace{50pt}
			\makebox[.1\textwidth][c]{\small  $2000$ tu}
			\hspace{50pt}
			\makebox[.1\textwidth][c]{\small  $4000$ tu}
			\hspace{20pt}
			\par\smallskip

			\begin{minipage}[t]{0.03\textwidth}
				\rotatebox{90}{\small{Truth}   $u$}
			\end{minipage}
			\begin{minipage}[t]{0.03\textwidth}
				\rotatebox{90}{Grid (240x240)}
			\end{minipage}
			\includegraphics[scale = 0.37]{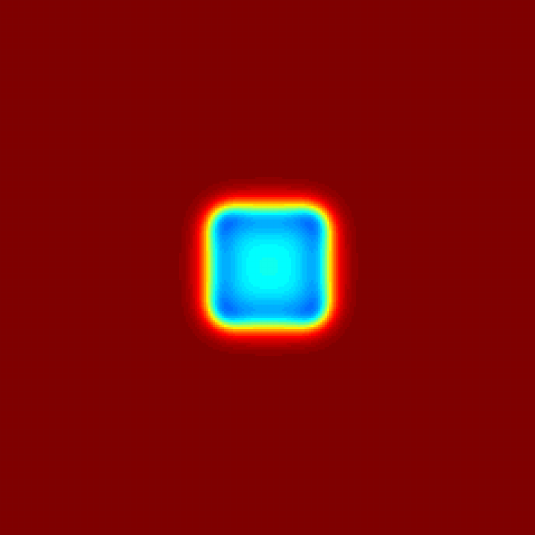}
			\includegraphics[scale = 0.37]{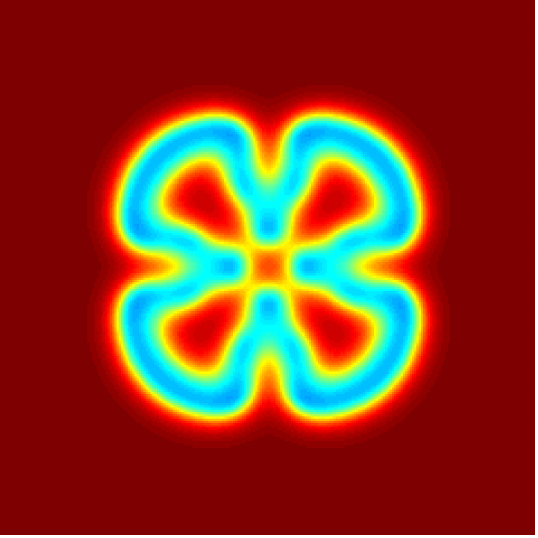}
			\includegraphics[scale = 0.37]{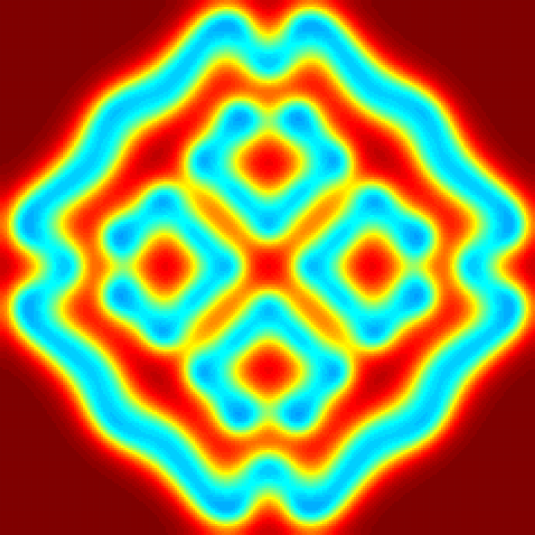}
			\includegraphics[scale = 0.37]{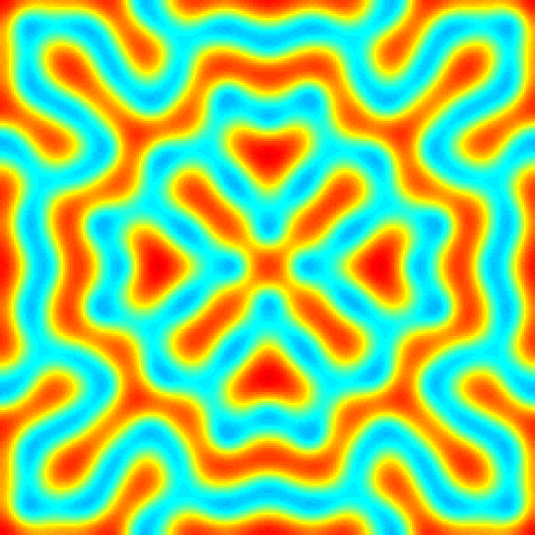}
		\end{subfigure}
		\\[5pt]
		
		\begin{subfigure}[t]{1\textwidth}
			\centering
			\begin{minipage}[t]{0.03\textwidth}
				\rotatebox{90}{\small{Truth}   $v$}
			\end{minipage}
			\begin{minipage}[t]{0.03\textwidth}
				\rotatebox{90}{Grid (240x240)}
			\end{minipage}
			\includegraphics[scale = 0.37]{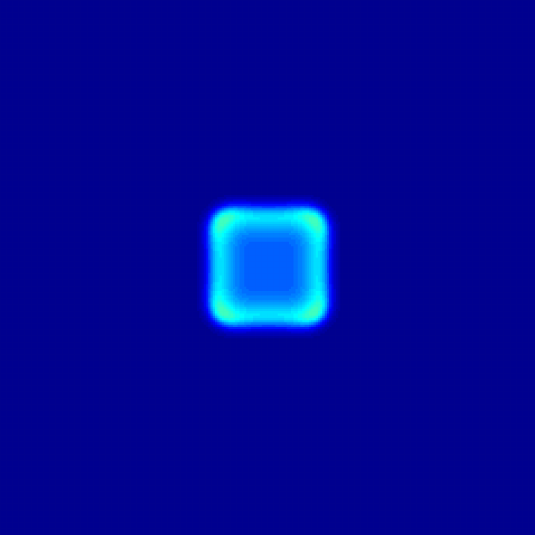}
			\includegraphics[scale = 0.37]{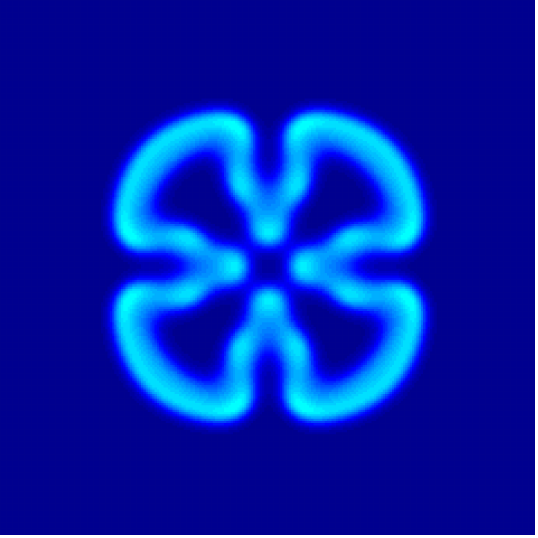}
			\includegraphics[scale = 0.37]{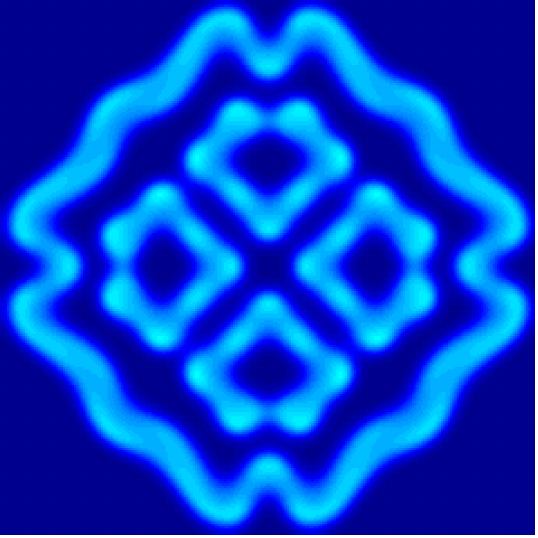}
			\includegraphics[scale = 0.37]{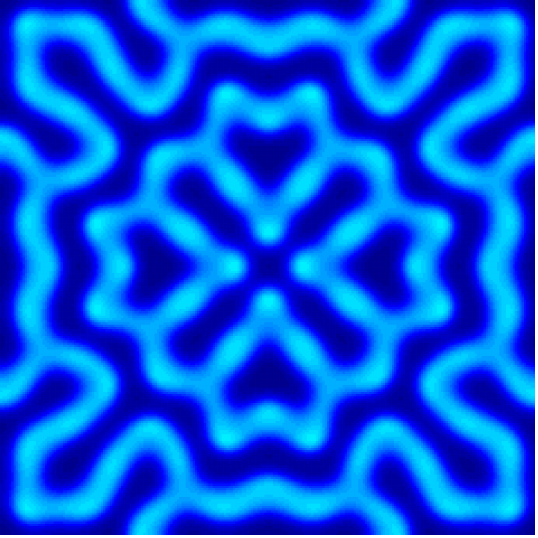}
		\end{subfigure}
		\\[5pt]
		\begin{subfigure}[t]{1\textwidth}
			\centering
			\begin{minipage}[t]{0.03\textwidth}
				\rotatebox{90}{\small{Observation} $\mI_H v$}
			\end{minipage}
			\begin{minipage}[t]{0.03\textwidth}
				\rotatebox{90}{Grid (24x24)}
			\end{minipage}
			\includegraphics[scale = 0.37]{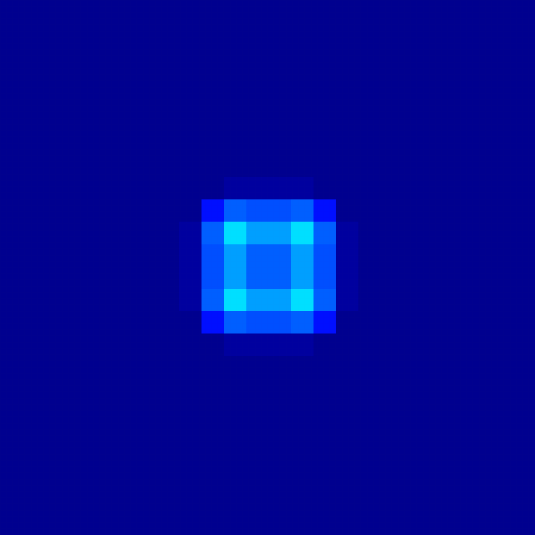}
			\includegraphics[scale = 0.37]{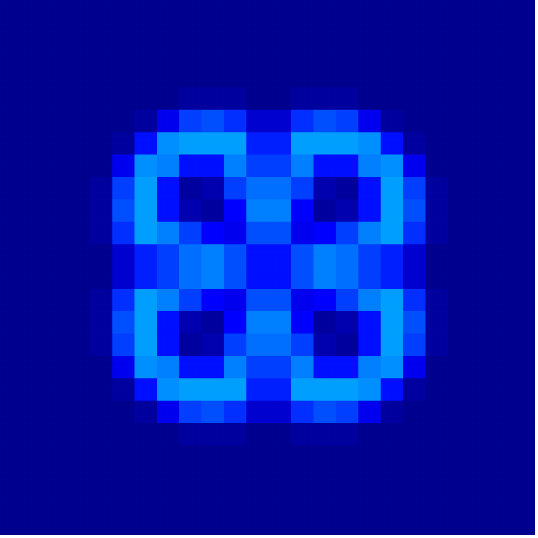}
			\includegraphics[scale = 0.37]{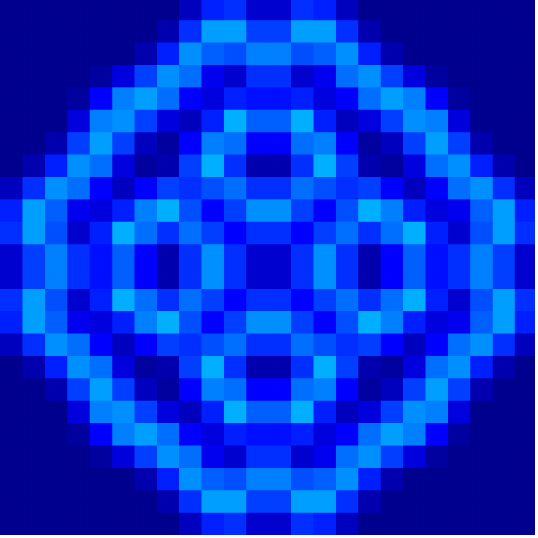}
			\includegraphics[scale = 0.37]{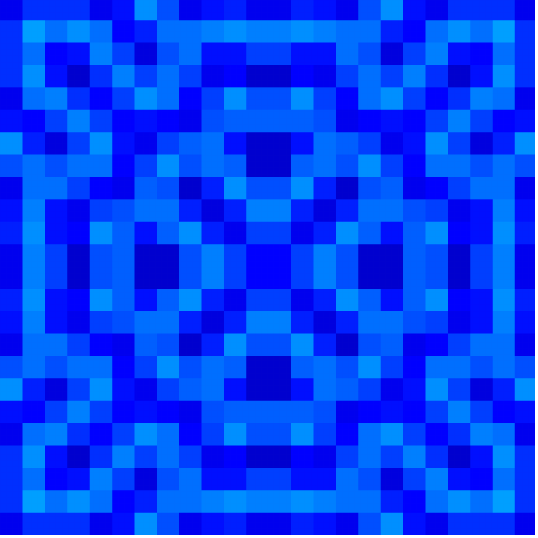}
		\end{subfigure}
		\hfill	
		\\[5pt]
		\begin{subfigure}[c]{1.\textwidth}
			\centering
			\includegraphics[scale = .75]{colorbar_jet.pdf}
		\end{subfigure}
		\hfill
		
		\caption{Snapshots of the reference Truth: Labyrinthine pattern ($F = 0.037$, $k = 0.060$), and the observations. Colormap shows concentration $u$, $v$, or $\mI_H v$.}
		\label{fig:labyrinthine}
	\end{figure}

	\begin{figure}[t]
		\begin{subfigure}[t]{1\textwidth}
			
			\centering
			\makebox[.1\textwidth][c]{}
			\hspace{0pt}
			\makebox[.1\textwidth][c]{\small  $100$ tu}
			\hspace{50pt}
			\makebox[.1\textwidth][c]{\small  $1000$ tu}
			\hspace{50pt}
			\makebox[.1\textwidth][c]{\small  $2000$ tu}
			\hspace{50pt}
			\makebox[.1\textwidth][c]{\small  $4000$ tu}
			\hspace{10pt}
			\par\smallskip

			\begin{minipage}[t]{0.03\textwidth}
				\rotatebox{90}{\small{Nudged}   $\tu$}
			\end{minipage}
			\begin{minipage}[t]{0.03\textwidth}
				\rotatebox{90}{Grid (240x240)}
			\end{minipage}
			\includegraphics[scale = 0.37]{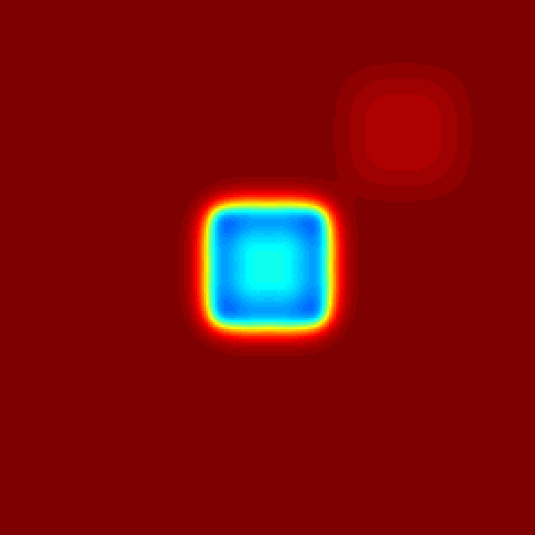}
			\includegraphics[scale = 0.37]{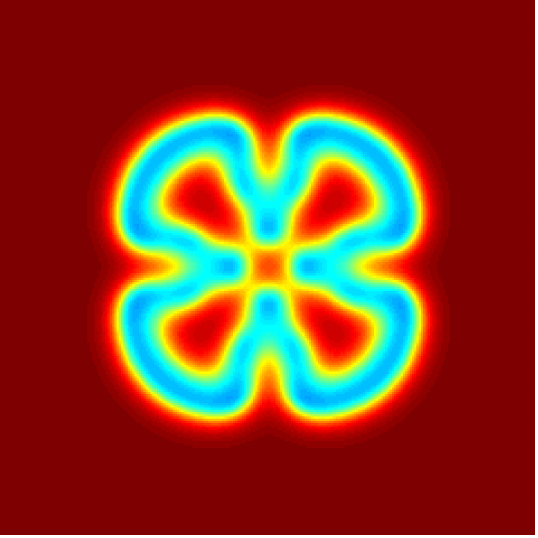}
			\includegraphics[scale = 0.37]{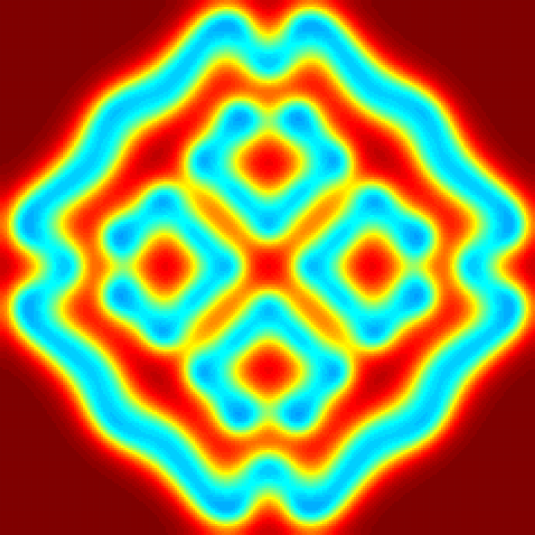}
			\includegraphics[scale = 0.37]{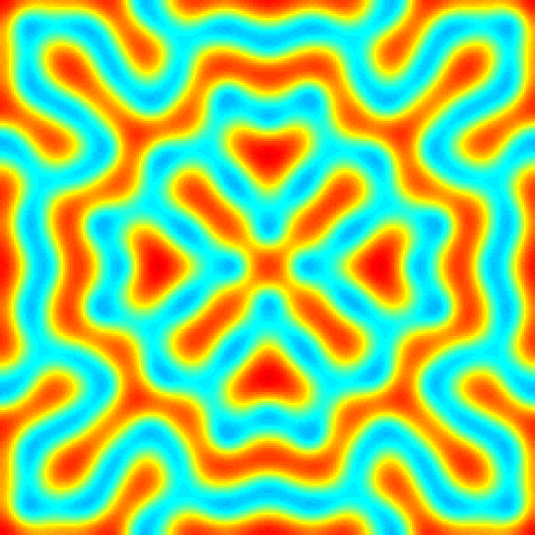}
		\end{subfigure}
		\\[5pt]
		\begin{subfigure}[t]{1\textwidth}
			\centering
			\begin{minipage}[t]{0.03\textwidth}
				\rotatebox{90}{\small{Nudged with delay}   $\tu$}
			\end{minipage}
			\begin{minipage}[t]{0.03\textwidth}
				\rotatebox{90}{Grid (240x240)}
			\end{minipage}
			\includegraphics[scale = 0.37]{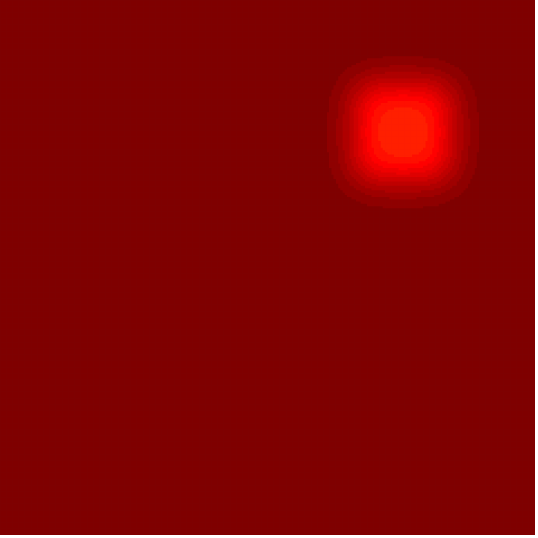}
			\includegraphics[scale = 0.37]{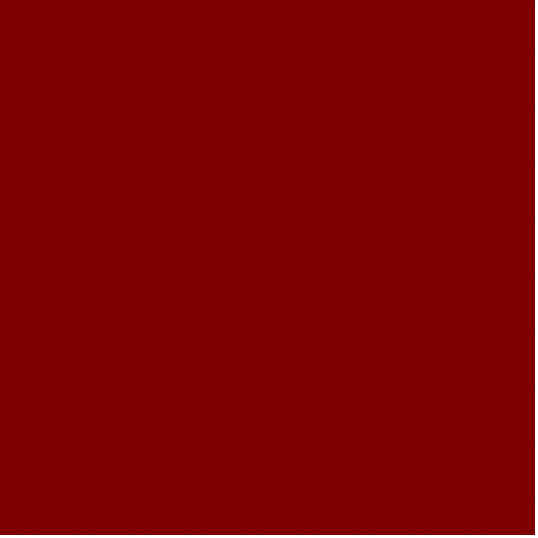}
			\includegraphics[scale = 0.37]{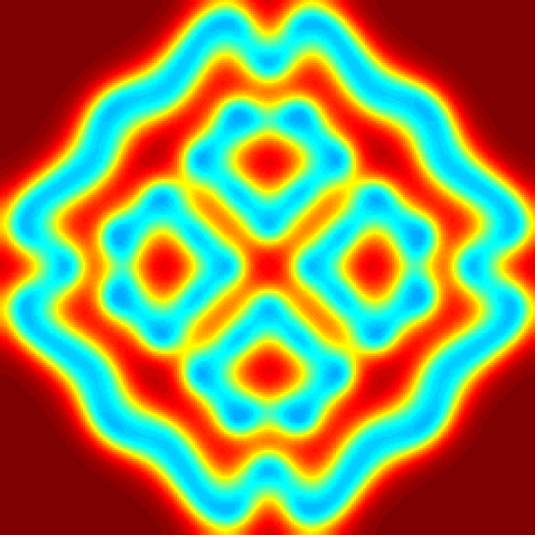}
			\includegraphics[scale = 0.37]{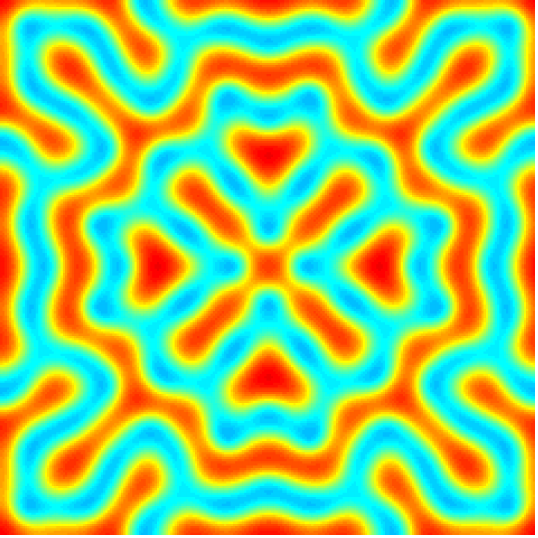}
		\end{subfigure}
		\hfill	
		\\[5pt]
		\begin{subfigure}[c]{1.\textwidth}
			\centering
			\includegraphics[scale = .75]{colorbar_jet.pdf}
		\end{subfigure}
		\hfill
		
		\caption{Snapshots of the reconstructed Labyrinthine pattern ($F = 0.037$, $k = 0.060$). Colormap shows concentration $\tu$.}
		\label{fig:rec_labyrinthine}
	\end{figure}
	
	Despite a poor initial guess, the nudged system gradually synchronizes with the Truth and recovers fine-scale structure; see \cref{fig:rec_labyrinthine}. Large-scale symmetry is captured quickly, after which assimilation steadily sharpens small-scale features until full synchronization. In addition, to avoid trivial synchronization due to coincident initial seeds, we also deliberately delay the correction and activate it only after 1000 time units (tu), see the second row of \cref{fig:rec_labyrinthine}. Here ``deactivating the correction'' means setting $\mu_v = 0$. While the correction is off, the nudged trajectory deviates from the Truth; once the correction is turned on ($\mu_v >0$), the observations nudge the dynamics back until synchronization is achieved.
	
		\begin{figure}[t]
		\begin{subfigure}[c]{1\textwidth}
			\centering
			\begin{minipage}[t]{0.03\textwidth}
				\rotatebox{90}{\qquad\small Relative $L^2$-error (logscale)}
			\end{minipage}
			\includegraphics[scale = 0.5]{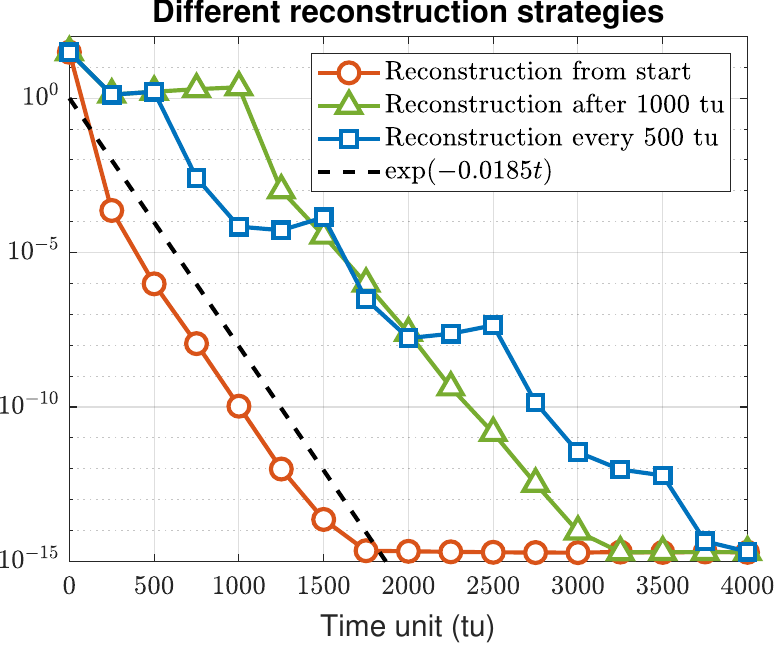}
			\quad
			\includegraphics[scale = 0.5]{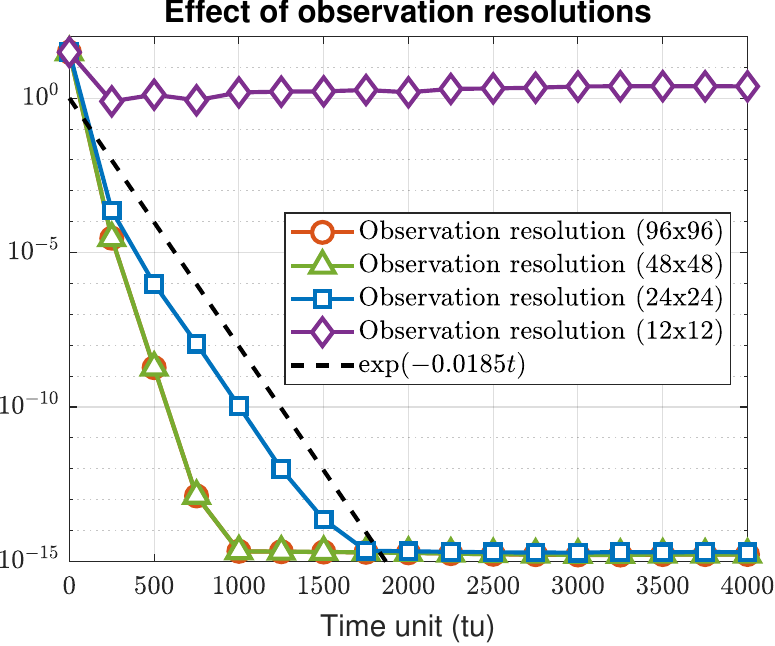}
			\\
			\hfill
			\\
			\begin{minipage}[t]{0.03\textwidth}
				\rotatebox{90}{\qquad\small Relative $L^2$-error (logscale)}
			\end{minipage}
			\includegraphics[scale = 0.5]{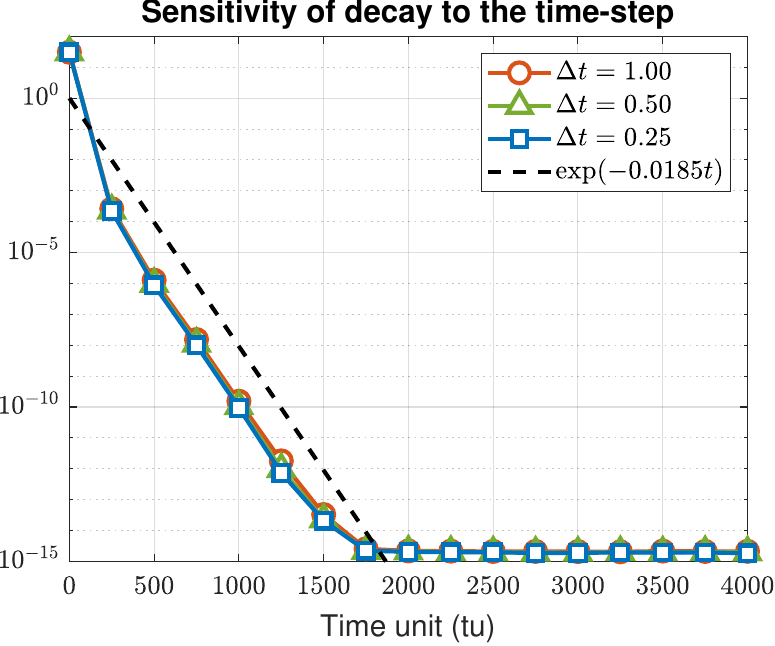}
			\quad
			\includegraphics[scale = 0.5]{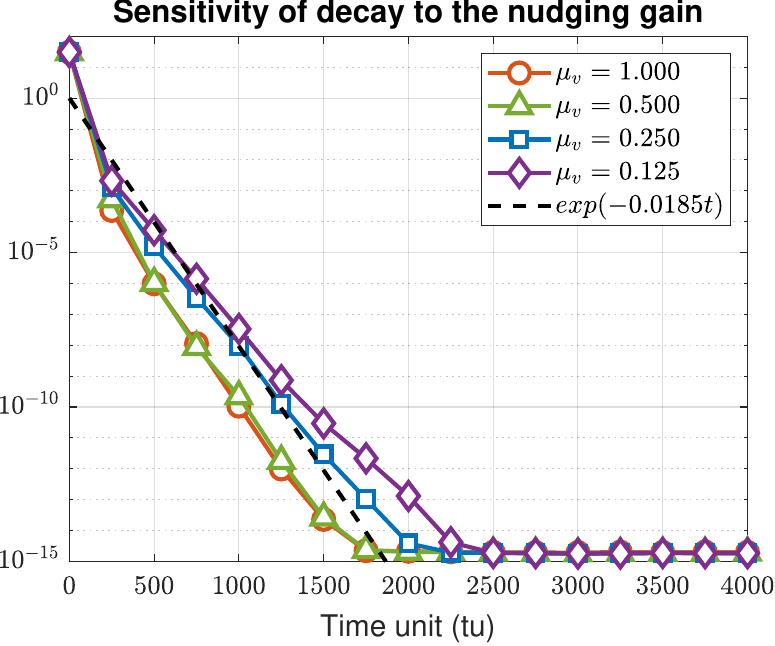}
		\end{subfigure}
		\caption{Relative $L^2$-error of the discrete data assimilation algorithm under different reconstruction strategies and configurations of the nudging gain $\mu_v$, time-step $\Delta t$, and observation resolution.}
		\label{fig:l2_error}
	\end{figure}
	
	We compare three strategies:
	(i) \emph{from start}: correction active from $t=0$ (used in the top row of \cref{fig:rec_labyrinthine});
	(ii) \emph{delayed}: correction off until 1000 tu and then on (bottom row of \cref{fig:rec_labyrinthine});
	(iii) \emph{periodic}: correction alternates, $500$ tu on and $500$ tu off.
	Activating the correction from the start yields the earliest synchronization (machine precision, $L^2$-error $\sim10^{-15}$). 
	Once the correction is on, all strategies exhibit essentially the same decay rate. 
	The periodic strategy is slower overall but uses only half as many temporal measurements (correction active $50\%$ of the time) while still producing the characteristic staircase decay, demonstrating robustness to temporally and spatially sparse observations.
	
	In practice, we seek the lowest possible observation resolution. However, \cref{thm:DDA_conv} indicates that if the resolution is too coarse relative to diffusion, coercivity may be lost. We fix $\mu_v=1$. As shown in \cref{fig:l2_error} (top right), a $12\times12$ observation grid is too sparse to steer the computed solution toward the Truth.	Refining the grid once yields synchronization at an acceptable rate; further refinement improves the decay rate with diminishing returns as the resolution becomes unnecessarily fine. The practical guideline is to choose the coarsest grid that preserves coercivity (as dictated by diffusion) while achieving reliable synchronization at the chosen $\mu_v$.
	
	Within the tested range, varying the time step $\Delta t$ preserves the decay rate of the relative $L^2$-error (\cref{fig:l2_error}, bottom left). Although not shown, instabilities are observed for $\Delta t>1$.
	
	\cref{fig:l2_error} (bottom right) shows that the nudging gain behaves as predicted by \cref{thm:DDA_conv}: increasing $\mu_v$ improves the decay rate. 
	In our experiments, doubling $\mu_v$ advances synchronization by roughly $500$ tu. 
	The benefit saturates around $\mu_v\approx1$, beyond which further increases bring no improvement. 
	The scheme becomes unstable for excessively large or small nudging gains, specifically for $\mu_v>2$ or $\mu_v<0.1$.

	\section{Conclusion}
	\label{sec:conclusion}
	We successfully applied the Azouani--Olson--Titi nudging framework to the Gray--Scott system and established two main results:
	(i) \emph{continuous synchronization}: the continuous data assimilation problem drives the assimilated solution toward the true solution with exponentially decaying error under suitable gains and observation resolution; and
	(ii) \emph{discrete synchronization}: the fully discrete finite-volume scheme enjoys the same behavior, with an additional mild step-size condition.
	
	We validated the theory with a set of targeted tests on the labyrinthine pattern. First, we confirmed synchronization from mismatched initial seeds and documented exponential decay of the $L^2$-error. Next, we studied how performance depends on the observation resolution $H$, the nudging gains $(\mu_u,\mu_v)$, and the correction frequency (time step $\Delta t$): finer $H$, larger gains, and more frequent updates yield faster convergence. We found a behavior \emph{not} predicted by our theory: the dynamics can be recovered when observing only the species $v$, i.e., with $\mu_u=0$. This suggests that indirect coupling through the reaction terms can suffice in practice. Even if we observe (and nudge) only the species $v$, the equations couple $u$ and $v$ strongly, so correcting $v$ forces $u$ to follow; the feedback applied to $v$ does not remain in $v$ alone, it pushes $u$ indirectly through the reaction terms, and diffusion then spreads these corrections in space. This explains why we still observe convergence to the truth when $\mu_u=0$.
	
	By contrast, the opposite setting did not yield the same result: taking $\mu_u>0$ and $\mu_v=0$ did not recover the dynamics. This mirrors related AOT results where observing the ``right'' variable can be enough, e.g., for 2D Navier--Stokes, recovering the flow from a single velocity component \cite{farhat2016abridged}, and for Bénard convection in porous media, recovering the state from temperature alone \cite{farhat2016data}. In our case, nudging $u$ alone could not control $v$, and synchronization typically failed. This asymmetry points to the central role of $v$ in driving the reaction term and suggests that $v$ is the more informative species to observe. It prompts two follow-up questions: (i) under what conditions can single-species observations guarantee synchronization, and (ii) in multi-species systems, such as the reduced mathematical models of the Belousov--Zhabotinsky chemical oscillator. How can we identify the dominant species whose observation most effectively steers the dynamics? We leave these as directions for future work.

	In this work we reconstructed Gray--Scott dynamics from sparse but noise-free observations. A natural next step is to study \emph{noisy} measurements and quantify how noise level affects synchronization. Another open question is \emph{partial spatial coverage}: can we still recover the dynamics if observations are available only on a subset of the domain, e.g., along the boundary or near the border.  
	
 	\section*{Declarations}
 	\subsection*{Conflicts of Interest} The author have no competing interests to declare that are relevant to the content of this article.
 	
 	\subsection*{Reproducibility} Inquiries about data availability should be directed to the author.
	
	% ==============================================================================
	% BIBTEX ======================================================================
	\bibliographystyle{abbrv}
	%bibliography style : abbrv, acm, alpha, apalike, ieeetr, plain, siam, unsrt
%	\bibliography{GS_finite_volume}

\end{document}